\documentclass[a4paper]{article}
\usepackage[utf8]{inputenc} 
\usepackage{amsfonts}
\usepackage{amsfonts,latexsym,amsmath,mathrsfs,amscd,geometry}
\usepackage{amssymb}
\usepackage{latexsym}
\usepackage{color}
\usepackage{upgreek}
\usepackage{appendix}
\usepackage{amsmath}
\numberwithin{equation}{section}
\usepackage{amsthm}
\usepackage{amsmath}
\usepackage{verbatim}
\usepackage{orcidlink}

\newcommand \nc{\newcommand}
\newtheorem{theorem}{Theorem}[section]
\newtheorem{lemma}[theorem]{Lemma}
\newtheorem{proposition}[theorem]{Proposition}
\newtheorem{corollary}[theorem]{Corollary}
\newtheorem{definition}[theorem]{Definition}

\newtheorem{remark}[theorem]{Remark}
\date{}

\nc{\R}{\mathbb{R}} \nc{\va}{\varepsilon} \nc{\ls}{\limits}
\def\D{\Delta}

\def\di{\mathrm{div\,}}

\newcommand \pp {\partial}
\newcommand{\dd}{\,\mathrm{d}}
\newcommand{\vphi}{\varphi}
\allowdisplaybreaks

\def\e{\varepsilon}

\begin{document}
\title{VANISHING ANGULAR VISCOSITY LIMIT FOR MICROPOLAR FLUID MODEL IN $\mathbb{R}_+^2$: BOUNDARY LAYER AND OPTIMAL CONVERGENCE RATE 
}
\author{ Yinghui Wang\orcidlink{0000-0002-7565-5525}\footnote{Corresponding author. MOE-LCSM, School of Mathematics and Statistics, Hunan Normal University, Changsha 410081, China. E-mail: yhwangmath@hunnu.edu.cn, yhwangmath@163.com.},\, Weihao Zhang\orcidlink{0000-0002-1361-2108}
\footnote{Zhuhai NO.1 High School, Zhuhai 519000, China. E-mail:weihao-zhang@qq.com.}
}

\maketitle

\begin{abstract}
We consider the initial-boundary value problem for the incompressible two-dimensional micropolar fluid model with angular viscosity in the  upper half-plane. This model  describes the motion of viscous fluids with microstructure.
 The global well-posedness of strong solutions for this problem with positive angular viscosity can be established via the standard energy method, as presented in the classical monograph [\L kaszewicz, {\it Micropolar fluids: Theory and applications.} Birkh\"auser, 1999]. Corresponding results for the zero angular viscosity case were established recently in [Liu, Wang, {\it Commun. Math. Sci.} 16 (2018), no. 8, 2147–2165]. However, the link between the positive angular viscosity model  (the full diffusive system) and the  zero angular viscosity model (the partially diffusive system) via the vanishing diffusion limit remains unknown. 
In this work,  we first construct Prandtl-type boundary layer profiles. We then provide a rigorous justification for the vanishing angular viscosity limit of global strong solutions, without imposing smallness assumptions on the initial data. Our analysis reveals the emergence of a strong boundary layer in the angular velocity field (micro-rotation velocity  of the fluid particles) during this vanishing viscosity process.
Moreover, we also obtain the optimal $L^\infty$ convergence rate as the angular viscosity tends to zero. Our approach combines anisotropic Sobolev spaces with careful energy estimates to address the nonlinear interaction between the velocity and angular velocity fields.

\end{abstract}

\bigbreak \textbf{{\bf Key Words}:}  Micropolar equations; Vanishing angular viscosity limit; Boundary layers.

\bigbreak  {\textbf{AMS Subject Classification 2020:}  35Q35; 76A05; 76D10; 76M45}

\section{Introduction} 
Micropolar fluid theory was introduced by Eringen (\cite{Eringen_1966,Eringen_1969}) in the 1960s to model complex fluids where the microstructure and intrinsic particle rotation significantly influence mechanical behavior. Unlike the classical Newtonian fluids,  the micropolar fluid model incorporates an angular velocity field $w$ to introduce additional rotational degrees of freedom beyond the standard translational motion. The framework describes diverse systems including: suspensions of randomly oriented particles, liquid crystals, polymeric fluids, and blood flow (capturing red blood cell rotation). It applies particularly to scenarios where micro-scale rotational inertia affects macroscopic behavior, such as small-scale flows or high-concentration suspensions.   For comprehensive applications, see Maugin \cite{Maugin_2017} and the reference therein for detailed discussion. The three-dimensional micropolar equations are given by: 
\begin{equation} \label{micropolar_equtions}
\begin{cases}  
\pp_t  u+\left(u\cdot\nabla \right)u+\nabla p-\left(\mu+\zeta \right)\D u=2\zeta\nabla\times w,\\
\pp_t  w +\left(u\cdot\nabla \right)w+4\zeta w-\nu\D w-\left( \nu+\lambda\right)\nabla\di w=2\zeta\nabla\times u,\\
\di u=0,
\end{cases}  
\end{equation} 
where $u=u(x,y,z,t)$ denotes the fluid velocity, $p(x,y,z,t)$ the pressure, $w=(x,y,z,t)$ the micro-rotation field (representing the angular velocity of the rotation of the fluid particles). The parameter $\mu\ge 0$ represents the Newtonian kinematic viscosity, $\zeta>0$ the micro-rotation viscosity, $\nu, \lambda\ge 0 $ the angular viscosity. 
Critically, the coupling between velocity and micro-rotation in micropolar fluids requires $\zeta>0$. When $\zeta=0$, the fields decouple and micro-rotation  ceases to influence the flow.

To reduce computational complexity while preserving  essential physics of micropolar fluid, researchers commonly adopt a 2D simplification of \eqref{micropolar_equtions}. Specifically, assuming that  $$u=\left(u_1(x,y,t),u_2(x,y,t),0\right),~ p=p(x,y,t),~w=\left(0,0,w(x,y,t)\right),$$ 
then, $\eqref{micropolar_equtions}$ reduces to 
\begin{equation} \label{u_mu}
\begin{cases}  
\pp_t  u+\left(u\cdot\nabla \right)u+\nabla p-\left(\mu+\zeta \right)\D u=-2\zeta\nabla^\perp w,\\
\pp_t  w +\left(u\cdot\nabla \right)w+4\zeta w-\nu\D w=2\zeta\nabla^\perp\cdot u,\\
\di u=0,
\end{cases}  
\end{equation}  
where  $u(x,y,t)$ is a 2D vector-valued function, $w(x,y,t)$ is a scalar function and  $\nabla^\perp:=(-\partial_y,\partial_x)^\top$.

Let us give a brief overview of some relevant works on the micropolar fluids. 
For the full dissipative case (i.e. $\nu>0$ in \eqref{micropolar_equtions}), in seminal works by Galdi and Rionero \cite{Galdi_Rionero_1977}, followed by \L ukaszewicz \cite{Lukaszewicz_1988}, the existence of weak solutions to system \eqref{micropolar_equtions} was established. Subsequently, in his monograph \cite{Lukaszewicz_1999}, \L ukaszewicz investigated the well-posedness of both weak and strong solutions for \eqref{micropolar_equtions} and its stationary counterpart. Moreover, \L ukaszewicz's monograph \cite{Lukaszewicz_1999} also surveys key mathematical advances in micropolar fluids throughout the 20th century.  
See also the works by Chen and Price \cite{Chen_Price_2006}, Dong and Chen \cite{Dong_Chen_2009} for the large time behaviors of the strong solutions of \eqref{micropolar_equtions} and  \eqref{uepsilon}, respectively. Recently, Chu and Xiao \cite{Chu_Xiao_2023_3D} studied the vanishing dissipation limits ($\mu,\zeta,\nu,\lambda\to 0$) of \eqref{micropolar_equtions} under slip boundary conditions over 3D bounded domains. Notably, the slip conditions and decoupled limit system enabled
justification of the vanishing dissipation limit via standard energy methods. However, justification of  vanishing dissipation limits for  Dirichlet initial boundary value problem remains open. Extensive research has also been conducted on the mathematical theory of compressible micropolar model and magneto-micropolar models; one can refer to recent works \cite{Cruz_et_al_2022,Feng_Hong_Zhu_2024,Qian_He_Zhang_2024,Tong_Pan_Tan_2021} and reference therein for details. 

Significant research also exists for the 2D model \eqref{u_mu}. The global existence of strong solutions was established by \L ukaszewicz \cite{Lukaszewicz_1999}.
For the vanishing angular viscosity case (i.e. $\nu= 0$ in \eqref{u_mu}), Dong and Zhang \cite{Dong_Zhang_2010} proved global well-posedness of strong solutions to the Cauchy problem of \eqref{u_mu}. Later, Liu and Wang \cite{Liu_Wang_2018} extended this result to Dirichlet initial-boundary value problems.
Chen, Xu and Zhang \cite{Chen_Xu_Zhang_2014} studied the simultaneous vanishing limit ($\nu=\zeta \to 0$) for weak solutions, demonstrating potential boundary effects when $\displaystyle \lim_{\nu\to 0} \zeta/\nu^{1/2} < \infty$.
Recently, Chu and Xiao \cite{Chu_Xiao_2023_2D} studied the vanishing micro-rotation and angular viscosity limit ($\nu=\zeta \to 0$) for \eqref{u_mu} under slip boundary conditions. Notice that  setting the micro-rotation viscosity $\zeta = 0$ in \eqref{u_mu} decouples equations \eqref{u_mu}$_1$ and \eqref{u_mu}$_2$, reducing the $u$-equation to Navier-Stokes equations. This indicates weakened $u$-$w$ interaction and loss of micro-rotational characteristics. Exploiting this decoupling, Chu and Xiao \cite{Chu_Xiao_2023_2D} showed no strong boundary layer effects emerge during the limit process, and the justification of the limit process can be established via standard energy method. A natural problem is that 
when preserving micro-rotational features ($\zeta>0$ fixed), under Dirichlet boundary conditions\footnote{In studies related to micropolar fluids, the Dirichlet boundary condition (i.e. the no-slip boundary condition) is more prevalent than the slip boundary condition, see \cite{Liu_Wang_2018,Lukaszewicz_1999} and the reference therein for details.}, can we rigorously justify the vanishing angular viscosity limit ($\nu\to 0$)? Does a strong boundary layer occur? Is there something different comparing to   \cite{Chu_Xiao_2023_2D}? The aim of this paper is to solve those problems. 

In another hand, for the inviscid case (omitting the viscous term $-(\mu+\zeta)\Delta u$ in \eqref{u_mu}), 
 the global well-posedness of strong solutions for  Cauchy problem  was established by Dong, Li and Wu \cite{Dong_Li_Wu_2017}. Simultaneously, Jiu, Liu, Wu and Yu \cite{Jiu_Liu_Wu_Yu_2017} proved analogous results for initial-boundary value problems with  boundary conditions: $u\cdot n |_{\partial \Omega} = w |_{\partial \Omega} = 0$.

While the well-posedness of the strong solutions for initial-boundary value problem of \eqref{u_mu}
 has been established for both positive angular viscosity \cite{Lukaszewicz_1999} and zero angular viscosity  \cite{Liu_Wang_2018}. The connection  between the diffusive ($\nu>0$) and the  non-diffusive  ($\nu=0$) models
remains unknown due to boundary layer effects. In this paper, we will provide a rigorous justification for the vanishing angular viscosity limit proocess.

\subsection{Reformulation of the problem}
We begin by setting $\mu = \nu =: \varepsilon$ in \eqref{u_mu} without loss of generality. Moreover, to minimize mathematical complexity, we consider \eqref{u_mu} on the upper half-plane $\mathbb{R}_{+}^2:=\{(x,y)\in\mathbb{R}^2|y>0\}$\footnote{Actually,  one can handle the case of  bounded smooth domains via  the boundary flattening technique (\cite{Evans_2010}).}. The system is then reformulated as follows:
\begin{equation} \label{uepsilon}
\begin{cases}  
\pp_t  u+\left(u\cdot\nabla \right)u+\nabla p-\left(\varepsilon+\zeta \right)\D u=-2\zeta\nabla^\perp w,\\
\pp_t  w +\left(u\cdot\nabla \right)w+4\zeta w-\varepsilon\D w=2\zeta\nabla^\perp\cdot u,\\
\di u=0. 
\end{cases}  
\end{equation}  
We equip \eqref{uepsilon} with the initial-boundary value conditions:
\begin{equation}\label{uepsilon_ini_bou_cond}
    \left( u,w\right)(x,y,0)=\left( u_0,w_0\right),~~\left( u,w\right)(x,0,t)=0.
\end{equation}
Formally, letting $\varepsilon\to 0$ in \eqref{uepsilon}, one can obtain the following zero angular viscosity model, 
\begin{equation} \label{I0}
\begin{cases}  
\pp_t  u^{I,0}+\left(u^{I,0}\cdot\nabla \right)u^{I,0}+\nabla p^{I,0}-\zeta\D u^{I,0}=-2\zeta\nabla^\perp w^{I,0},\\
\pp_t  w^{I,0} +\left(u^{I,0}\cdot\nabla \right)w^{I,0}+4\zeta w=2\zeta\nabla^\perp\cdot u^{I,0},\\
\di u^{I,0}=0, 
\end{cases}  
\end{equation}  
equipped with the initial-boundary value conditions:
\begin{equation}\label{I0_cond}
    \left( u^{I,0},w^{I,0}\right)(x,y,0)=\left( u_0 ,w_0 \right),~~u^{I,0}(x,0,t)=0.
\end{equation}
Crucially, a mismatch exists between the boundary conditions for $w$ in \eqref{uepsilon_ini_bou_cond} and $w^{I,0}$ in \eqref{I0_cond} at $\{y=0\}$. After a careful analysis via the asymptotic matching method (see Appendix \ref{Appendix A} for details), we can
find a boundary layer profile $w^{b,0}(x,\frac{y}{\sqrt{\varepsilon}},t)$ such that the solution to problem \eqref{uepsilon}-\eqref{uepsilon_ini_bou_cond} admit the asymptotic representation:
\begin{gather}\label{eq_Prandtl_exp}
    u(x,y,t) = u^{I,0}(x,y,t) + O(\varepsilon^{\frac{1}{2}}),~~~w(x,y,t) = w^{I,0}(x,y,t) +w^{b,0}(x,\frac{y}{\sqrt{\varepsilon}},t) + O(\varepsilon^{\frac{1}{2}}),
\end{gather}
where $O(\varepsilon^{\frac{1}{2}})\to 0$ in $L^\infty$-norm as $\varepsilon\to 0$. In the rest of the present paper, we will validate the boundary layer expansion \eqref{eq_Prandtl_exp} and  rigorously justify the  vanishing angular viscosity limit for the initial-boundary value problem \eqref{uepsilon}-\eqref{uepsilon_ini_bou_cond}.  
\subsection*{Notation} 
We introduce the following notation conventions.  Let $C$ denote a generic positive constant depending on the initial data and fixed parameters, but independent of the variable parameter $\varepsilon$. When emphasizing dependence, we use subscripts such as $C_\zeta$. Some other notations are defined as follows:
	\begin{itemize}
		\item $A \lesssim B\iff A\le CB.$
		\item For a scalar-valued function $w$, two vector-valued functions $u$ and $v$, we use the following notations:
		\begin{align*}
          \nabla^\perp w := (-\partial_y w,\partial_x w)^\top,~
		(\nabla w)_i:=\pp_iw,~(\nabla u)_{ij}:=\pp_ju_i,~\nabla^\perp\cdot u := \partial_xu_2-\partial_y u_1,~
		(u\otimes v)_{ij}:=u_iv_j.
		\end{align*}
		\item $\langle \cdot\rangle :=\sqrt{1+|\cdot|^2}.$
		\item $L^p_{xy}$ and $H^s_{xy}$ denote the usual Lebesuge and Sobolev space over $\mathbb{R}^2_+$  with corresponding norms $\| \cdot \|_{L^p_{xy}}$
		and $\| \cdot \|_{H^s_{xy}}$, respectively.
  \item The notation $\langle \cdot,\cdot\rangle$ means the $L^2$ inner product over $\mathbb{R}^2_+:=
		\left\{ (x,y)| (x,y) \in \mathbb{R} \times \mathbb{R}_+\right\}$.
		\item The anisotropic Sobolev space is denoted as
		\begin{align*}
		H^m_x H^\ell_y :=
		\left\{
		f \in L^2(\mathbb{R}^2_+)
		\big|\sum_{
			0\leq i\leq m,0\leq j\leq \ell}
		\|\pp_x^i\pp_y^jf(x,y)\|_{L^2_{xy}} < \infty
		\right\},
		\end{align*}
		with norm $\| \cdot \|_{H^m_
			x H^\ell_y}$.
		\item $\overline{ f}:=f(x,0,t).$
		\item  $z:=\frac{y}{\sqrt{\e}}$ for $\e>0$. The notations $L^p_{xz}, H^s_{xz}$ and $H^m_xH^\ell_z$ denote that their components are functions of $(x,z)$.
		\item $\|(u,v)\|^2_X:=\|u\|^2_X+\|v\|^2_X$ for Banach space $X$. The norm of $L^q(0,T;X) (1\le p\le\infty)$ is denoted by $\|\cdot\|_{L^q_TX}$.
		\item  Let $\vphi$ be a smooth non-negative function defined on $[0,+\infty)$ satisfying
		\begin{align}\label{cut-off function}
		\vphi(0)=1, ~\vphi^\prime(0)=0, ~\vphi(z)=0 \text{ ~~for~~} z>1.
		\end{align}
	\end{itemize}

\subsection{Main result}
We begin with the global well-posedness of problem  \eqref{I0}-\eqref{I0_cond}.
\begin{proposition}\label{prop1}
    Assume that $\left( u_0,w_0\right) \in H_{xy}^{18}$ with $\di u_0  = 0$ and that the compatibility conditions
\begin{align}
    \pp_t^i u^{I,0}(0)\Big|_{y=0}=0,~~0\le i\le 8,
\end{align}
hold, where $ \pp_t^i u^{I,0}(0)$ is the i-th time derivative of $u^{I,0}$
at $\{ t = 0\}$ which can be connected
to the initial value $( u_0,w_0)$  by means of the system $\eqref{I0}$. Then, for given $T>0$, problem  \eqref{I0}-\eqref{I0_cond}  has a unique solution $\left( u^{I,0},w^{I,0}\right)$ over $[0,T]$ satisfying $\di u^{I,0}=0$ and
\begin{gather*}
     \pp_t^\ell u^{I,0}\in L^\infty(0,T;H^{18-2\ell}_{xy})\cap L^2(0,T;H^{19-2\ell}_{xy}),~~\ell=0,1,\cdots,9,\\
     w^{I,0}\in L^\infty(0,T;H^{18}_{xy}),~\pp_t^jw^{I,0}\in L^\infty(0,T;H^{19-2j}_{xy}),~~j=1,2,\cdots,9.
\end{gather*}
\end{proposition} 
\begin{remark}
    The proof of Proposition \ref{prop1} follows standard arguments. Global well-posedness for \eqref{I0}-\eqref{I0_cond} with $H^2$ initial data over bounded smooth domains was established in \cite{Liu_Wang_2018}. After some slight modification, one can easily prove  the well-posedness of problem \eqref{I0}-\eqref{I0_cond}. Moreover, the higher regularity follows by standard induction arguments (cf. Chapter 7 of $\cite{Evans_2010}$). 
\end{remark}
The well-posedness of problem \eqref{uepsilon}-\eqref{uepsilon_ini_bou_cond} is stated as follows.
\begin{proposition}\label{prop2}
    Assume that $\left( u_0,w_0\right) \in H_{xy}^{2}$ with $\di u_0   = 0$ and that the compatibility conditions
\begin{align}
    (u_0,w_0)\Big|_{y=0}=0,~~(\pp_t  u,\pp_t w)(0)\Big|_{y=0}=0,
\end{align}
hold, where $ (\pp_t  u(0),\pp_tw(0))$ is the time derivative of $(  u , w )$
at $\{ t = 0\}$ which can be connected
to the initial value $\left( u_0 ,w_0 \right)$  by means of the problem \eqref{uepsilon}-\eqref{uepsilon_ini_bou_cond}. Then,  problem \eqref{uepsilon}-\eqref{uepsilon_ini_bou_cond} has a unique solution $\left( u ,w \right)$ satisfying $\di u =0$ and
\begin{align*}
    & (u,w)\in L^\infty(0,T;H^{2}_{xy})\cap L^2(0,T;H^{3}_{xy}),~~~(\partial_t u,\partial_t w)\in L^\infty(0,T;L^{2}_{xy})\cap L^2(0,T;H^{1}_{xy}).
\end{align*}
\end{proposition}
\begin{remark}
    Due to the present of diffusion term $-\nu\Delta w $ in \eqref{uepsilon}$_2$ (which makes the equation of $w$ a parabolic PDE), the well-posedness of system \eqref{uepsilon}-\eqref{uepsilon_ini_bou_cond} can be established  via standard energy method. In \cite{Nong_etal_2025}, the authors proved the global well-posedness of some 2D Oldroyd-B models which have a more complicated nonlinear structure than \eqref{uepsilon}. One can adopt the analytical framework of  \cite{Nong_etal_2025} to prove Proposition \ref{prop2}. Here, we omit the details for brevity.
\end{remark}
We are now prepared to present our main result.
\begin{theorem}\label{thm1}
    In addition to the conditions of Proposition \ref{prop1} and \ref{prop2}, we assume 
    the $(u_0,w_0)$ satisfies the additional strong compatibility conditions \eqref{compa_strong}, \eqref{compatibility condition of I1}, \eqref{compatibility condition of I2}.
    Then, the following uniform estimates 
    \begin{gather}
         \left\| u(x,y,t)-u^{I,0}(x,y,t) \right\|_{L^\infty_TL^\infty_{xy}}\le C\e^\frac{1}{2}, \label{eq_convergence_u}\\[2mm]
         \left\| w(x,y,t)-w^{I,0}(x,y,t)-w^{b,0}\left(x,\frac{y}{\sqrt{\e}},t\right) \right\|_{L^\infty_TL^\infty_{xy}}\le C\e^\frac{1}{2},\label{eq_convergence_w}
    \end{gather}
hold, where the positive constant $C$ is independent of $\e$ and $w^{b,0}$ is the solution of problem $\eqref{wb0}$.
\end{theorem}
\begin{remark}
    The results in Theorem \ref{thm1} suggest that the strong boundary layer effect happens to $w$ in the sense of \eqref{eq_convergence_w} but not to $u$ (see \eqref{eq_convergence_u}).
\end{remark}
\begin{remark}
   Theorem \ref{thm1} requires no smallness assumptions on either the time interval $[0,T]$ or initial data $(u_0,w_0)$.
\end{remark}
\begin{remark}
    In the recent work \cite{Chu_Xiao_2023_2D}, the authors studied the limit process $\zeta = \nu \to 0$ for \eqref{u_mu} equipped with the initial-boundary conditions:
    \begin{gather} 
        (u,w)|_{t=0} = (u_0,w_0),~~~\text{in}~ \Omega, \label{eq_con_Chu_in}\\ 
        u\cdot n = 0,~~~\nabla^\perp\cdot u = 0,~~w = 0,~~~\text{on}~\partial \Omega,\label{eq_con_Chu}
    \end{gather}
    over 2D smooth bounded domain $\Omega$. Under some suitable assumptions, they proved that when $\nu,\zeta\to 0$, the solution of problem \eqref{u_mu}, \eqref{eq_con_Chu_in} and \eqref{eq_con_Chu} converge to the solution of the following problem
    \begin{equation} \label{u_mu_0}
        \begin{cases}  
        \pp_t  u+\left(u\cdot\nabla \right)u+\nabla p- \mu \D u= 0,\\
        \pp_t  w +\left(u\cdot\nabla \right)w = 0,\\
        \di u=0,
        \end{cases}  
    \end{equation}  
    \begin{gather} 
        (u,w)|_{t=0} = (u_0,w_0),~~~\text{in}~ \Omega, \label{eq_con_Chu_in_1}\\ 
        u\cdot n = 0,~~~\nabla^\perp\cdot u = 0,~~~\text{on}~\partial \Omega.\label{eq_con_Chu_1}
    \end{gather}
    Moreover, they proved that the strong solution of \eqref{u_mu_0}-\eqref{eq_con_Chu_1} satisfy $w = 0$ on $\partial \Omega$\footnote{The main observation here is that \eqref{u_mu_0}$_2$ is a transport equation for $w$. Then, due to the boundary condition $u|_{\pp \Omega} = 0$, one can easily show that $w|_{\pp \Omega} = 0$. See the proof of Theorem 3 in \cite{Chu_Xiao_2023_2D} for details.}, which coincides with the boundary condition \eqref{eq_con_Chu} of problem \eqref{u_mu}. Hence, there is no strong boundary layer in the limit process from problem \eqref{u_mu}, \eqref{eq_con_Chu_in} and \eqref{eq_con_Chu} to problem \eqref{u_mu_0}-\eqref{eq_con_Chu_1}. In comparison to the results of \cite{Chu_Xiao_2023_2D}, our results in Theorem \ref{thm1} show that there exists a strong boundary layer in the vanishing angular viscosity process. The main reason for this difference is the strong coupling of $u^{I,0}$ and $w^{I,0}$ in \eqref{I0}, see the analysis in Section \ref{Asymptotic analysis} for details.
\end{remark} 
The convergence rates in \eqref{eq_convergence_u} and \eqref{eq_convergence_w} are optimal. Actually, under the assumptions of Theorem \ref{thm1}, the order of   boundary layer thickness is close to the value $O(\varepsilon^{\alpha})(0<\alpha< 1/2)$. To illustrate this, we recall the definition of boundary layer thickness.
\begin{definition}[\cite{Frid_Shelukhin_1999}]\label{Def}
Let $(u,w) $ and $(u^{I,0},w^{I,0}) $ be the solutions of problems \eqref{uepsilon}-\eqref{uepsilon_ini_bou_cond} and \eqref{I0}-\eqref{I0_cond}, respectively. If there is a non-negative function $\delta=\delta(\e)$ satisfying $\delta(\e) \to 0$ as $\e \to 0$ such that
    \begin{align*} 
        \liminf\limits_{\e\to 0}\left\| \left( u-u^{I,0},w-w^{I,0}\right) \right\|_{L^\infty(0,T ;L^\infty(\mathbb{R}^2_+))}>0,
\end{align*}
and
    \begin{align*} 
        \lim\limits_{\e\to 0}\left\| \left( u-u^{I,0},w-w^{I,0}\right) \right\|_{L^\infty(0,T; L^\infty(\mathbb{R}\times(\delta,+\infty)))}=0.
\end{align*}
Then, we say that the initial-boundary value problem  \eqref{uepsilon}-\eqref{uepsilon_ini_bou_cond}  has a non-trivial boundary layer solution as $\e \to 0$, and $\delta(\varepsilon)$ is called a boundary layer thickness (BL-thickness) for problem  \eqref{uepsilon}-\eqref{uepsilon_ini_bou_cond}.
\end{definition}
\begin{remark}
    From Definition \ref{Def}, one can easily find that any function $\tilde{\delta}(\varepsilon)$ satisfying $\tilde{\delta}(\varepsilon) \geq \delta(\varepsilon)$ for small $\varepsilon$ is also a BL-thickness. Thus, the BL-thickness is not unique.
\end{remark}
For the BL-thickness of the problem  \eqref{uepsilon}-\eqref{uepsilon_ini_bou_cond}, we have  the following result.
\begin{theorem}\label{thm2}
    Under the assumptions of Theorem \ref{thm1}, let $\delta(\varepsilon)$ be a smooth function of $\varepsilon>0$ satisfying $\delta(\varepsilon)\downarrow 0$ and $\varepsilon^{\frac{1}{2}}/\delta(\varepsilon)\to 0$ as $\varepsilon \downarrow 0$. 
    Then, $\delta(\varepsilon)$ is a BL-thickness of problem  \eqref{uepsilon}-\eqref{uepsilon_ini_bou_cond}, such that
 \begin{align}\label{liminf}
        \liminf\limits_{\e\to 0}\left\| \left( u-u^{I,0},w-w^{I,0}\right) \right\|_{L^\infty(0,T ;L^\infty(\mathbb{R}^2_+))}>0,
\end{align}
and
    \begin{align}\label{lim}
        \lim\limits_{\e\to 0}\left\| \left( u-u^{I,0},w-w^{I,0}\right) \right\|_{L^\infty(0,T; L^\infty(\mathbb{R}\times(\delta,+\infty)))}=0,
\end{align}
if and only if
$$  
w^{I,0}(x,0,t)\neq 0,~~~\text{for some}~~t\in [0,T].
$$
\end{theorem}
\subsection{Main ideas}
Unlike the results in \cite{Chu_Xiao_2023_2D}, the rigorous justification of the limit process from \eqref{uepsilon}-\eqref{uepsilon_ini_bou_cond} to \eqref{I0}-\eqref{I0_cond} as $\varepsilon\to 0$ faces a fundamental challenge: the strong coupling between $u^{I,0}$ and $w^{I,0}$ induces a  mismatch between $w$ and $w^{I,0}$ at the boundary. Our core methodology addresses this by constructing boundary layer correctors. Formally, using the method of matched asymptotic expansions (refer, for instance, to Chapter 4 in $\cite{Holmes_2013}$; see also $\cite{Hou_Wang_2019}$), we decompose the solution to problem \eqref{uepsilon}-\eqref{uepsilon_ini_bou_cond} as follows:
\begin{align*}
\displaystyle u(x,y,t)=\,&u^{I,0}(x,y,t)+\e^\frac{1}{2}u^{I,1}(x,y,t)+\e u^{I,2}(x,y,t)\\
&+\e^\frac{1}{2}u^{b,1}\left(x,\frac{y}{\sqrt{\e}},t\right)+\e u^{b,2}\left(x,\frac{y}{\sqrt{\e}},t\right)+\e^\frac{3}{2}u^{b,3}\left(x,\frac{y}{\sqrt{\e}},t\right)\\
&+\e^2(0,u^{b,4}_2)^\top\left(x,\frac{y}{\sqrt{\e}},t\right)+\e^\frac{3}{2}S(x,y,t)+U^\e(x,y,t),\\
\displaystyle p(x,y,t)=\,&p^{I,0}(x,y,t)+\e^\frac{1}{2}p^{I,1}(x,y,t)+\e p^{I,2}(x,y,t)+P^\e(x,y,t),\\
\displaystyle w(x,y,t)=\,&w^{I,0}(x,y,t)+\e^\frac{1}{2}w^{I,1}(x,y,t)+\e w^{I,2}(x,y,t)\\
&+w^{b,0}\left(x,\frac{y}{\sqrt{\e}},t\right)+\e^\frac{1}{2}w^{b,1}\left(x,\frac{y}{\sqrt{\e}},t\right)+\e w^{b,2}\left(x,\frac{y}{\sqrt{\e}},t\right)\\
&+W^\e(x,y,t),
\end{align*} 
where $(u^{I,0},p^{I,0},w^{I,0})$ is the solution to the limit problem  \eqref{I0}-\eqref{I0_cond}. And the higher order profiles  $(u^{I,i},p^{I,i},w^{I,i}),~(u^{b,j}, w^{b,j})$ are given via asymptotic matched expansion, see Section \ref{Asymptotic analysis}. Due to the influence of $2\zeta\nabla^{\perp}\cdot u^{I,0}$ in \eqref{I0}$_2$, in general $w^{I,0}(x,0,t) \neq 0$ for certain $t\in (0,T]$, which leads to a non-trivial boundary layer profile $w^{b,0}$ (see Lemma \ref{regularity of wb0ub11} for details).
Furthermore, thanks to the specific structure of systems \eqref{uepsilon} and  \eqref{I0}, the higher order profiles $(u^{I,i},p^{I,i},w^{I,i})$ and $(u^{b,i}, w^{b,i}),~(i\geq 1)$  satisfy  linear equations. These equations can be solved sequentially, incorporating the initial-boundary conditions. Additionally, these profiles exhibit good regularity and decay properties in the spatial variable, see Section $\ref{Regularity of boundary layer profiles}$.
Consequently, the problem reduces to estimating the remainder terms $(U^\e,P^\e,W^\e)$, which relies on several technical estimates detailed in Section \ref{Estimates_error}.

The rest of the paper is organized as follows. In Section $\ref{Some inequalities}$, we introduce some preliminary results that will be used in the proof. In Section $\ref{Construction of an approximate solution}$, the boundary layer equation is constructed via asymptotic analysis and the well-posedness of the boundary layer profiles is obtained by the energy method. In Section $\ref{Proof of the theorem}$, we construct the error equations and estimate the source terms. Then, the uniform estimates of the error terms are deduced in sequel. Then, we complete the proof of Theorem \eqref{thm1} and \ref{thm2} in Section $\ref{Convergence_rate}$ and \ref{sec_bl}. Finally, in Appendix $\ref{Appendix A}$, we give the details of the
asymptotic analysis.  Appendix $\ref{Appendix B}$ provides the full expressions for relevant source terms appearing in the proof.

\section{Preliminaries}\label{Some inequalities}
In this section, we recall some useful results which will be used later. We begin with the following  inequalities.
\begin{lemma}
    Assume that $f,g,h,\in H^1_{xy}$. Then the inequality 
    \begin{align}\label{eq_Lady}
    \|f\|_{L^4_{xy}}\le C\|f\|^\frac{1}{2}_{L^2_{xy}}  \|\nabla f\|^\frac{1}{2}_{L^2_{xy}},
    \end{align}
    hold.
\end{lemma}
\begin{remark}
    Inequality \eqref{eq_Lady} is the 2D Ladyzhenskaya inequality, see \cite{Seregin_2014} for a proof. 
    
\end{remark}
\begin{lemma}[Hardy inequality (\cite{Hardy_Littlewood_Polya_1952})]
         If $1<p<\infty$ and $f\in L^p(0,\infty)$, then $f\in L^1(0,\infty)$ and 
    \begin{align}
        \int_0^{+\infty}\left(\frac{\int_0^yf(t)\dd t}{y}\right)^p\dd y\le C_p\int_0^{+\infty} |f(y)|^p\dd y.
    \end{align}
\end{lemma}

Next,
we introduce the following Dirichlet  problem for heat equation which
will be used to analyze the well-posedness of the boundary layer profiles, i.e.
\begin{equation}\label{parobolic equation}
    \begin{cases}
        \pp_t\theta(x,z,t)-\pp_z^2\theta(x,z,t)=g^b(x,z,t),\\
        \theta(x,0,t)=0, ~\theta(x,z,0)=0.
    \end{cases}
\end{equation}
  Then, we have the following well-posedness results.
\begin{proposition}[Proposition 3.1 in \cite{Hou_Wang_2019}]\label{regularity of b}
    For given $T\in(0,+\infty)$ and  $m\in\mathbb{N}_+$. Assume that 
    \begin{align*}
        \langle z\rangle^\ell \pp_t^i g^b\in L^2(0,T;H^{2m-2i}_xL^2_z),~~i=0,1,\cdots,m,
    \end{align*}
    where $\ell\in\mathbb{N}$, $g^b$ satisfies the compatibility condition
    \begin{align*}  
    \pp_t^kg^b\Big|_{z=0}=0,~k=0,1,\cdots,m-1,
    \end{align*}
 for \eqref{parobolic equation}. Then,  \eqref{parobolic equation} admits a  unique solution $\theta(x,z,t)$ on $[0,T]$ satisfies 
\begin{gather*}
     \langle z\rangle^\ell \pp_t^i\theta\in L^\infty(0,T;H^{2m-2i}_xH^1_z)\cap L^2(0,T;H^{2m-2i}_xH^2_z),~i=0,1,\cdots,m.
\end{gather*} 
\end{proposition}
\begin{remark}
    The well-posedness of problem \eqref{parobolic equation} is standard; one can refer, for instance, to $\cite{Lady Solo and Ural}$. The regularity stated in Proposition 3.1 can be referred to $\cite{Wang and Wen}$ and the references therein.
\end{remark}
We also need the following linearized problem to  analyze the well-posedness of the higher-order outer layer profiles.
\begin{equation}\label{outer equation}
    \begin{cases}
        \pp_t \tilde{u}+(\tilde{u}\cdot\nabla)a+(a\cdot\nabla)\tilde{u}+\nabla\tilde{p}-\zeta\D\tilde{u}-2\zeta\nabla^\perp\tilde{w}=\tilde{f},\\
        \pp_t \tilde{w}+(\tilde{u}\cdot\nabla)b+(a\cdot\nabla)\tilde{w}+4\zeta\tilde{w}-2\zeta\nabla^\perp\cdot\tilde{u}=\tilde{g},\\
        \di \tilde{u}=0,\\
        \tilde{u}(x,0,t)=0, ~(\tilde{u},\tilde{w})(x,y,0)=(\tilde{u}_0,\tilde{w}_0).
    \end{cases}
\end{equation}
The well-posedness of \eqref{outer equation} is stated as follows.
\begin{proposition}\label{regularity of I}
        For given $T\in(0,+\infty)$ and $m\in \mathbb{N}_+$, assume that
    \begin{gather*}
         \di a=0,~a(x,0,t)=0,~(\tilde{u}_0,\tilde{w}_0)\in H^{2m+1}_{xy},\\
         \pp_t^ia,~\pp_t^ib\in L^\infty_TL^{2m+1-2i}_{xy}\cap L^2_TH^{2m+2-2i}_{xy},~i=0,1,\cdots,~m,\\
         \pp_t^m\tilde{f}\in L^2_TL^2_{xy},~\pp_t^j\tilde{f}\in L^\infty_TH^{2m-1-2j}_{xy}\cap L^2_TH^{2m-2j}_{xy},~j=0,1,\cdots,m-1,\\
         \pp_t^i\tilde{g}\in L^\infty_TH^{2m-2j}_{xy}\cap L^2_TH^{2m+1-2i}_{xy},~i=0,1,\cdots,m,
    \end{gather*}
    and further that  $a,b,\tilde{f}$ and $\tilde{g}$ satisfies the compatibility conditions up to order $m$ for problem \eqref{outer equation}, i.e.,
    \begin{align*}
        \pp_t^\ell \tilde{u}(0)\Big|_{y=0}=0,~\ell=0,1,\cdots,m.
    \end{align*}
Then,  problem $\eqref{outer equation}$ admits a unique solution $(\tilde{u},\tilde{w})$ satisfying 
\begin{gather*}
     \pp_t^{m+1}\tilde{u}\in L^2_TL^2_{xy},~\pp_t^i \tilde{u}\in L^\infty_TH^{2m+1-2i}_{xy}\cap L^2_TH^{2m+2-2i}_{xy},~i=0,1,\cdots,~m,\\
     \tilde{w}\in L^\infty_TH^{2m+1}_{xy},~\pp_t^k\tilde{w}\in L^\infty_TH^{2m+2-2k}_{xy},~k=1,2,\cdots,m+1.
\end{gather*}
\end{proposition}
\begin{remark}
Since the system $\eqref{outer equation}$ is linear, with Proposition \ref{prop1} in hand, the proof of Proposition $\ref{regularity of I}$ is standard for the higher regularity of the given coefficients and initial data. One can refer to \cite{Wang and Wen} and the reference therein for related discuss. We omit the details here for brevity.
\end{remark}
\section{Construction of an approximate solution}\label{Construction of an approximate solution}

In this section, we present the equations of outer and inner (i.e. boundary) layer profiles via asymptotic analysis whose derivations will be given in Appendix \ref{Appendix A}. Based on the outer and inner layer profiles, we can construct an approximate solution to the problem  \eqref{uepsilon}-\eqref{uepsilon_ini_bou_cond}  which is used to prove
Theorem $\ref{thm1}$.
 
\subsection{Asymptotic analysis}\label{Asymptotic analysis}
In this subsection, we derive the equations of the outer and inner layer profiles by  the asymptotic analysis (see, e.g. \cite{Wang and Wen}). To begin with, we introduce the following Prandtl type boundary layer expansions: 
\begin{equation}\label{expansions}
    \begin{cases}\displaystyle
        u(x,y,t)=\sum_{j=0}^{+\infty}\e^\frac{j}{2}\left(u^{I,j}(x,y,t)+u^{b,j}(x,z,t)\right),\\ \displaystyle
        p(x,y,t)=\sum_{j=0}^{+\infty}\e^\frac{j}{2}\left(p^{I,j}(x,y,t)+p^{b,j}(x,z,t)\right),\\ \displaystyle
        w(x,y,t)=\sum_{j=0}^{+\infty}\e^\frac{j}{2}\left(w^{I,j}(x,y,t)+w^{b,j}(x,z,t)\right),
    \end{cases}
\end{equation}
where $z=y/\sqrt{\e}$. 
We assume that
\begin{align*}
    u^{b,j}(x,z,t)\to 0,~~p^{b,j}(x,z,t)\to 0,~~w^{b,j}(x,z,t)\to 0,
\end{align*}
fast enough as $z \to +\infty$ . Substituting $\eqref{expansions}$ in to \eqref{uepsilon}-\eqref{uepsilon_ini_bou_cond}, and applying the matched asymptotic method, we can deduce the equations of outer and inner layer profiles in sequence, see Appendix  \ref{Appendix A}   for the details. 
\subsubsection{The leading order inner and outer profiles}
Due to the analysis in Lemma \ref{lem_b0} (see Appendix  \ref{Appendix A}), we find that 
the leading order outer layer profile $(u^{I,0},p^{I,0},w^{I,0})$ satisfies problem  \eqref{I0}-\eqref{I0_cond}. Moreover, the leading boundary layer profiles of velocity and pressure satisfy
$$u^{b,0} = 0,~p^{b,0} = 0.$$  For the angular velocity, there is non-trivial boundary layer, i.e. $w^{b,0}$ is the solution of following problem
\begin{equation}\label{wb0}
    \begin{cases}
        \pp_tw^{b,0}-\pp_z^2 w^{b,0}=0,\\[1mm]
         w^{b,0}(x,z,0)=0,~w^{b,0}(x,0,t)=-w^{I,0}(x,0,t).
    \end{cases}
\end{equation} 
\subsubsection{The first order inner and outer profiles}
From Corollary \ref{coro_u_b_1} and Lemma \ref{lem_b1}, we find that 
    \begin{equation}\label{eq_u_b_1_p}
    u^{b,1}_1 = 2\int_z^{+\infty} w^{b,0}(x,s,t)\mathrm{d}s,~~~u^{b,1}_2 = 0,~~p^{b,1} = 0. 
    \end{equation}
Then, the outer profiles $\left(u^{I,1},w^{I,1},p^{I,1} \right)$ satisfy the following
problem:
\begin{equation}\label{I1}
    \begin{cases}
        \pp_t  u^{I,1}+\left( u^{I,1}\cdot\nabla\right)u^{I,0}+\left( u^{I,0}\cdot\nabla\right)u^{I,1}+\nabla p^{I,1}-\zeta\D u^{I,1}+2\zeta\nabla^\perp w^{I,1}=0,\\ 
        \pp_t  w^{I,1}+\left( u^{I,1}\cdot\nabla\right)w^{I,0}+\left( u^{I,0}\cdot\nabla\right)w^{I,1}+4\zeta w^{I,1}-2\zeta\nabla^\perp\cdot u^{I,1}=0,\\
        \di u^{I,1}=0,\\[1mm]
        \left(u^{I,1},w^{I,1}\right)(x,y,0)=0,\\ 
        \displaystyle u_{1}^{I,1} (x,0,t)= -2\int_0^{+\infty} w^{b,0}(x,s,t)\mathrm{d}s,~~~u_{2}^{I,1} (x,0,t)= 0.
    \end{cases}
\end{equation} 
Furthermore,  $w^{b,1}$ satisfies
\begin{equation}\label{wb1}
    \begin{cases}
        \pp_tw^{b,1}-\pp_z^2 w^{b,1}=g^{b,1},\\
        w^{b,1}(x,0,t)=-w^{I,1}(x,0,t),~ w^{b,1}(x,z,0)=0.
    \end{cases}
\end{equation}
where 
\begin{align*}
    -g^{b,1}=&\,\overline{u^{I,1}_1}\pp_x w^{b,0}+u^{b,1}_1\overline{\pp_x w^{I,0}}+u^{b,1}_1\pp_xw^{b,0}\\
    &+\left(\overline{ u^{I,2}_2}+u^{b,2}_2 \right)\pp_zw^{b,0}+\frac{1}{2}\overline{\pp_y^2u^{I,0}_2}z^2\pp_zw^{b,0}\\
&+\overline{\pp_yu^{I,0}_1}z\pp_xw^{b,0}+\overline{\pp_yu^{I,1}_2}z\pp_zw^{b,0}.
\end{align*}
See {\bfseries Step 3} in Appendix \ref{Appendix A} for details.

\subsubsection{The second order inner and outer profiles}
From Corollary \ref{coro_u_b_2} and Lemma \ref{lem_b2}, we find 
    the second order boundary layer profile $u^{b,2}$ and $p^{b,2}$ satisfy 
    \begin{equation}\label{eq_u_p_b_2}
    u^{b,2}_1 = 2\int_z^{+\infty} w^{b,1}(x,s,t)\mathrm{d}s,~~~u^{b,2}_2 =   2\int_z^\infty \int_\tau^{+\infty} \partial_xw^{b,0}(x,s,t)\mathrm{d}s\mathrm{d}\tau,~~~p^{b,2} = 0. 
    \end{equation}
The second order outer profiles $\left(u^{I,2},w^{I,2},p^{I,2} \right)$  satisfy the following problem:
\begin{equation}\label{I2}
    \begin{cases}
        \pp_t  u^{I,2}+\left( u^{I,2}\cdot\nabla\right)u^{I,0}+\left( u^{I,0}\cdot\nabla\right)u^{I,2}+\nabla p^{I,2}-\zeta\D u^{I,2}+2\zeta\nabla^\perp w^{I,2}=f^{I,2},\\ 
        \pp_t  w^{I,2}+\left( u^{I,2}\cdot\nabla\right)w^{I,0}+\left( u^{I,0}\cdot\nabla\right)w^{I,2}+4\zeta w^{I,2}-2\zeta\nabla^\perp\cdot u^{I,2}=g^{I,2},\\
        \di u^{I,2}=0,\\
        u^{I,2}(x,0,t)= -2\begin{pmatrix}
                \int_0^{+\infty} w^{b,1}(x,s,t)\mathrm{d}s \\  \int_0^\infty \int_\tau^{+\infty} \partial_xw^{b,0}(x,s,t)\mathrm{d}s\mathrm{d}\tau
        \end{pmatrix},~~\left(u^{I,2},w^{I,2}\right)(x,y,0)=0.
    \end{cases}
\end{equation}
where 
\begin{align*}
     f^{I,2}=-\left(u^{I,1}\cdot\nabla\right)u^{I,1}+\D u^{I,0}, ~~
     g^{I,2}=-\left(u^{I,1}\cdot\nabla\right)w^{I,1}+\D w^{I,0}.
\end{align*}
Furthermore, $w^{b,2}$ satisfies 
\begin{equation}\label{wb2}
    \begin{cases}
        \pp_tw^{b,2}-\pp_z^2 w^{b,2}=g^{b,2},\\
        w^{b,2}(x,0,t)=-w^{I,2}(x,0,t),~ w^{b,2}(x,z,0)=0,
    \end{cases}
\end{equation}
where
\begin{align*} 
       - g^{b,2} =\,& 
         \overline{u^{I,2}_1}\pp_x w^{b,0}+u^{b,2}_1\overline{ \pp_xw^{I,0}}+u^{b,2}_1\pp_xw^{b,0} +\overline{u^{I,1}_1}\pp_x w^{b,1}+u^{b,1}_1\overline{ \pp_xw^{I,1}}+u^{b,1}_1\pp_xw^{b,1}\\
&+u^{b,2}_2\overline{\pp_yw^{I,0}}+\left(\overline{u^{I,2}_2}+u^{b,2}_2 \right)\pp_zw^{b,1} -2\zeta\pp_xu^{b,2}_2  -2 \partial_z  u_1^{b,1}-\pp_x^2w^{b,0}+\overline{\pp_yu^{I,1}_1}z\pp_xw^{b,0}\\
& +u^{b,1}_1\overline{\pp_y\pp_xw^{I,0}}z+\overline{\pp_yu^{I,0}_1}z\pp_xw^{b,1}+\overline{\pp_yu^{I,2}_2}z\pp_zw^{b,0}+\overline{\pp_yu^{I,1}_2}z\pp_zw^{b,1}\\
&
+\frac{1}{6}\overline{\pp_y^3u^{I,0}_2}z^3\pp_zw^{b,0}+\frac{1}{2}\overline{\pp_y^2u^{I,0}_1}z^2\pp_xw^{b,0}+\frac{1}{2}\overline{\pp_y^2u^{I,0}_2}z^2\pp_zw^{b,1}+\frac{1}{2}\overline{\pp_y^2u^{I,1}_2}z^2\pp_zw^{b,0} \\
&  -\left(2\int_0^{z}\int_\tau^{\infty} \partial_xw^{b,1}(x,s,t)\mathrm{d}s\mathrm{d}z\right)\pp_zw^{b,0}  + 2\int_z^\infty\left( \zeta \partial_x^2u_1^{b,1}  -  \partial_t u^{b,1}_1\right)\mathrm{d}z . 
\end{align*}
See {\bfseries Step 4} in Appendix \ref{Appendix A} for details. 
\subsubsection{Some higher order profiles}
We also need the following higher order profiles of velocity:   
    \begin{gather}\label{eq_u_b3_1}
        u^{b,3}_1 = 2\int_z^{\infty}w^{b,2} \mathrm{d}z - \frac{1}{\zeta} u_1^{b,1} - \frac{1}{\zeta}\int_z^{\infty}  \int_\tau^\infty\left( \zeta \partial_x^2u_1^{b,1}  -  \partial_t u^{b,1}_1\right)\mathrm{d}\tau \mathrm{d}z,\\ 
        u^{b,3}_2 =  2\int_z^\infty \int_\tau^{+\infty} \partial_xw^{b,1}(x,s,t)\mathrm{d}s\mathrm{d}\tau,
    \end{gather}
    and 
\begin{equation}\label{eq_u_b4_2} 
           u^{b,4}_2 = 2\int_z^{\infty}\int_\tau^\infty \partial_x w^{b,2}\mathrm{d}\tau \mathrm{d}z - \frac{1}{\zeta}\int_z^{\infty} \partial_x u_1^{b,1}\mathrm{d}z - \frac{1}{\zeta}\int_z^{\infty}  \int_\tau^\infty \int_s^\infty\left( \zeta \partial_x^3u_1^{b,1}  -  \partial_t\partial_x u^{b,1}_1\right)\mathrm{d}s\mathrm{d}\tau \mathrm{d}z.  
\end{equation}
See Lemma \ref{lem_b_higher} for details.

\subsection{Regularity of the outer and boundary layer profiles}\label{Regularity of boundary layer profiles}
In order to use the outer and inner layer profiles deduced in Section \ref{Asymptotic analysis}, we   prove the well-posedness of those profiles. To prove the higher order regularities of
the inner layer profiles, we need the following strong compatibility conditions:
\begin{equation}\label{compa_strong}
    \partial_t^j w^{I,0}(0)|_{y=0} = 0,~~~j=1,\cdots,8.
\end{equation}
Then, for $w^{b,0}$ and $u^{b,1}$, we have the following results.
\begin{lemma}\label{regularity of wb0ub11}
    Under the assumptions of Theorem $\ref{thm1}$, there exists a unique solution  $w^{b,0}$ of the problem  $\eqref{wb0}$ on $[0,T]$, satisfying
    \begin{align*}
        \langle z\rangle^\ell\pp_t^j w^{b,0}\in L^\infty(0,T;H^{16-2j}_xH^1_z)\cap L^2(0,T;H^{16-2j}_xH^2_z),
    \end{align*}
    for all $\ell\in\mathbb{N}$ and $j=0,1\cdots,8$. Furthermore, using \eqref{eq_u_b_1_p}, we have 
    \begin{align*} 
        \langle z\rangle^\ell\pp_t^j u^{b,1}_1\in L^\infty(0,T;H^{16-2j}_xH^2_z)\cap L^2(0,T;H^{16-2j}_xH^3_z).
    \end{align*}
\end{lemma}

\begin{proof}
    Denoting $\hat{w}^{b,0}(x,z,t)=w^{b,0}(x,z,t)+\vphi(z)w^{I,0}(x,0,t)$ with $\varphi$ defined in \eqref{cut-off function}, and using $\eqref{wb0}$, we have
    \begin{equation}\label{eq_w_b0_m}
        \begin{cases}
            \pp_t \hat{w}^{b,0} -\pp_z^2\hat{w}^{b,0} =\hat{r}^{b,0},\\
            \hat{w}^{b,0}(x,0,t)=0,~\hat{w}^{b,0}(x,z,0)=0.
        \end{cases}
    \end{equation}
    where
    \begin{align*}
        \hat{r}^{b,0}(x,z,t)=\vphi(z)\pp_t \overline{w^{I,0}} -\pp_z^2\vphi(z)\overline{w^{I,0}}.
    \end{align*}
    Noticing that, for $h(x,y,t)\in H^{k+1}_{xy}$ with fixed $t\in\mathbb{R}_+$ and $k\in\mathbb{N}$, we have 
    \begin{align*}
         \|\overline{h}\|_{H^k_x}^2 =&\, \sum_{i = 0}^k \int_{\mathbb{R}} |\partial_x^ih(x,0,t)|^2\mathrm{d} x \leq \sum_{i = 0}^k \int_{\mathbb{R}} \|\partial_x^ih(x,y,t)\|_{L^\infty_y}^2\mathrm{d} x \\ 
         \lesssim&\, \sum_{i = 0}^k \int_{\mathbb{R}} \|\partial_x^ih(x,y,t)\|_{H^1_y}^2\mathrm{d} x \lesssim   \|\partial_x^ih(x,y,t)\|_{H^{k+1}_{xy}}^2. 
    \end{align*}
 Then, using Proposition \ref{prop1}, we have, for $j=0,1,\cdots,8$ and $\ell\in\mathbb{N}$, that 
    \begin{align*}
        \left\| \langle z\rangle^\ell \pp_t^j\hat{r}^{b,0}\right\|_{L^2_TH^{16-2j}_xL^2_z}\lesssim&\,   \left\| \pp_t^j\pp_t\overline{w^{I,0}}\right\|_{L^2_TH^{16-2j}_{x}}   \left\| \langle z\rangle^{\ell }\vphi(z) \right\|_{L^2_z} +\left\| \pp_t^j\overline{w^{I,0}}\right\|_{L^2_TH^{16-2j}_{x}}\left\| \langle z\rangle^\ell\pp_z^2 \vphi(z)  \right\|_{L^2_z}\\
        \lesssim&\,\left\| \pp_t^{j+1}w^{I,0}\right\|_{L^2_TH^{19-2(j+1)}_{xy}}+\left\| \pp_t^{j}w^{I,0}\right\|_{L^2_TH^{17-2j}_{xy}} \le C.
    \end{align*} 
    Moreover, from \eqref{compa_strong}, one can easily check that $\hat{r}^{b,0}$ satisfies the eighth order compatibility condition stated in Proposition $\ref{regularity of b}$. Hence, using the Proposition $\ref{regularity of b}$ with $m=8$, we find that  problem  \eqref{eq_w_b0_m}  admits a unique solution $\hat{w}^{b,0}$ satisfying
    \begin{align*}
        \langle z\rangle^\ell \pp_t^j\hat{w}^{b,0}\in L^\infty(0,T;H^{16-2j}_xH^1_z)\cap L^2(0,T;H^{16-2j}_xH^2_z),~~j=0,1,\cdots,8.  
    \end{align*}
    Therefore,  \eqref{wb0} admits a unique solution $w^{b,0}$ satisfying
    \begin{align*}
        \langle z\rangle^\ell \pp_t^jw^{b,0}\in L^\infty(0,T;H^{16-2j}_xH^1_z)\cap L^2(0,T;H^{16-2j}_xH^2_z),~~j=0,1,\cdots,8. 
    \end{align*}
    Next, using \eqref{eq_u_b_1_p}, we have  
    \begin{align*}
        \left\| \langle z\rangle^\ell\pp_t^ju^{b,1}_1\right\|_{L^\infty_TH^{16-2j}_{x}H^2_z}^2=&~4\left\| \langle z\rangle^\ell\pp_t^j\int_z^{+\infty}w^{b,0}\dd z\right\|_{L^\infty_TH^{16-2j}_{x}H^2_z}^2\\
        \lesssim&~\left\| \langle z\rangle^\ell\pp_t^j\int_z^{+\infty}w^{b,0}\dd z\right\|_{L^\infty_TH^{16-2j}_{x}L^2_z}^2+\left\| \langle z\rangle^\ell\pp_t^j w^{b,0}\right\|_{L^\infty_TH^{16-2j}_{x}H^1_z}^2.
    \end{align*}
For the first term on the right hand side of above inequality, we have 
\begin{align*}
    &\left\| \langle z\rangle^\ell\pp_t^j\int_z^{+\infty}w^{b,0}\dd z\right\|_{L^\infty_TH^{16-2j}_{x}L^2_z}^2\\ 
    \lesssim&~ \left\|\sum_{j^\prime=0}^{16-2j}\int_0^{+\infty}\frac{1}{\langle z\rangle^4} \langle z\rangle^{2\ell+4}\left(\int_z^{+\infty}\left\|\pp_t^j\pp_x^{j^\prime} w^{b,0}(x,\eta,t)\right\|_{L^2_x}\dd \eta\right)^2\dd z\right\|_{L^\infty_T}\\
    \lesssim&~\left\|\sum_{j^\prime=0}^{16-2j}\int_0^{+\infty}\frac{1}{\langle z\rangle^4} \left(\int_0^{+\infty}\left\|\langle \eta\rangle^{\ell+2}\pp_t^j\pp_x^{j^\prime} w^{b,0}(x,\eta,t)\right\|_{L^2_x}\dd \eta\right)^2\dd z\right\|_{L^\infty_T}\\
    \lesssim&~\left\|\langle z\rangle^{\ell+2}\pp_t^jw^{b,0} \right\|_{L^\infty_TH^{16-2j}_{x}L^2_z}^2.
\end{align*}
Hence, we have 
\begin{align*}
    \left\| \langle z\rangle^\ell\pp_t^ju^{b,1}_1\right\|_{L^\infty_TH^{16-2j}_{x}H^2_z}^2\lesssim\left\| \langle z\rangle^{\ell+2}\pp_t^j w^{b,0}\right\|_{L^\infty_TH^{16-2j}_{x}H^1_z}^2.
\end{align*}
Similarly, we can prove 
    \begin{align*}
        \langle z\rangle^\ell\pp_t^j u^{b,1}_1\in L^\infty(0,T;H^{16-2j}_xH^2_z)\cap L^2(0,T;H^{16-2j}_xH^3_z),~~j=0,1,\cdots,8. 
    \end{align*} 
    The proof is complete.
\end{proof} 
To prove  higher order regularity of the first order outer layer profiles $(u^{I,1},w^{I,1})$, we need the following compatibility  conditions on the initial data: 
        \begin{align}\label{compatibility condition of I1}
        \pp_t^\ell u^{I,1}(0) \Big|_{y=0}=0,~\ell=0,1,\cdots,6,
    \end{align}
    which can be represented by $(u^{I,0},w^{I,0})$ via \eqref{I0} and \eqref{I1}. Here, we omit the details for brevity. One can refer to Remark 3.7 in \cite{Wang and Wen} for some related discuss.
\begin{lemma}\label{regularity of I1}
    Under the assumptions of Theorem $\ref{thm1}$, there exists a unique solution  $(u^{I,1},w^{I,1})$ of the problem  $\eqref{I1}$,   satisfying
    \begin{gather*}
         \pp_t^7 u^{I,1}\in L^2_TL^2_{xy},~ \pp_t^i u^{I,1}\in L^\infty_TH^{13-2i}\cap L^2_TH^{14-2i}, i=0,1\cdots,6,\\
         w^{I,1}\in L^\infty_TL^{13}_{xy},~ \pp_t^j w^{I,1}\in L^\infty_TL^{14-2j}_{xy}, j=1,2,\cdots,7.
    \end{gather*}
\end{lemma}

\begin{proof}
    Denoting \begin{align*}
        \hat{u}^{I,1}(x,y,t)=&~u^{I,1}(x,y,t)+\left(\begin{matrix}
        2\left[\varphi^2(y)+\varphi^\prime(y)\int^{y}_0\varphi(\xi)\dd \xi\right]  \int_0^{+\infty} w^{b,0}(x,s,t)\mathrm{d}s\\
        -2\varphi(y)\int^{y}_0\varphi(\xi)\dd \xi  \int_0^{+\infty} \pp_xw^{b,0}(x,s,t)\mathrm{d}s
    \end{matrix}\right)\\
    =:&~u^{I,1}(x,y,t)+S^{I,1}(x,y,t).
    \end{align*}
    and using $\eqref{I1}$, we have
    \begin{equation*}
    \begin{cases}
        \pp_t  \hat{u}^{I,1}+\left( \hat{u}^{I,1}\cdot\nabla\right)u^{I,0}+\left( u^{I,0}\cdot\nabla\right)\hat{u}^{I,1}+\nabla \hat{p}^{I,1}-\zeta\D \hat{u}^{I,1}-2\zeta\nabla^\perp w^{I,1}=\hat{f}^{I,1},\\ 
        \pp_t  w^{I,1}+\left( \hat{u}^{I,1}\cdot\nabla\right)w^{I,0}+\left( u^{I,0}\cdot\nabla\right)w^{I,1}+4\zeta w^{I,1}-2\zeta\nabla^\perp\cdot \hat{u}^{I,1}=\hat{g}^{I,1},\\
        \di \hat{u}^{I,1}=0,\\
        \hat{u}^{I,1}(x,0,t)=0,~\left(\hat{u}^{I,1},w^{I,1}\right)(x,y,0)=0.
    \end{cases}
\end{equation*}
where
\begin{gather*}
    \hat{f}^{I,1}=\pp_tS^{I,1}+S^{I,1}\cdot\nabla u^{I,0}+u^{I,0}\cdot\nabla S^{I,1}-\D S^{I,1},\\
     \hat{g}^{I,1}=S^{I,1}\cdot\nabla w^{I,0}-2\zeta\nabla^\perp\cdot S^{I,1}.
\end{gather*}
Using Proposition \ref{prop1} and Lemma \ref{regularity of wb0ub11}, a direct calculation yields that
\begin{gather*}
     \pp_t^6 \hat{f}^{I,1}\in L^2_TL^2_{xy}, ~ \pp^j_t\hat{f}^{I,1}\in L^{\infty}_TH^{13-2j}_{xy}\cap L^2_TH^{13-2j}_{xy}, ~j=0,\cdots,5,\\
     \pp^k_t\hat{g}^{I,1}\in L^{\infty}_TH^{14-2k}_{xy}\cap L^2_TH^{14-2k}_{xy}, ~k=0,\cdots,6.
\end{gather*}
Hence, using Proposition $\ref{regularity of I}$ with codition \eqref{compatibility condition of I1}, we can finish the proof of Lemma \ref{regularity of I1}. 
\end{proof}
Next, for the first order inner layer profiles, we have the following Lemma.  
\begin{lemma}\label{regularity of wb1ub21}
    Under the assumptions of Theorem $\ref{thm1}$,  there exists a  unique solution  $w^{b,1}$ of the problem  $\eqref{wb1}$, satisfying
    \begin{align*}
        \langle z\rangle^\ell\pp_t^j w^{b,1}\in L^\infty(0,T;H^{11-2j}_xH^1_z)\cap L^2(0,T;H^{11-2j}_xH^2_z) , 
    \end{align*}
    for all $\ell\in\mathbb{N}$ and $j=0,1\cdots,5$. Moreover, using \eqref{eq_u_p_b_2}, we have 
    \begin{gather*} 
        \langle z\rangle^\ell\pp_t^i u^{b,2}_1\in L^\infty(0,T;H^{11-2i}_xH^2_z)\cap L^2(0,T;H^{11-2i}_xH^3_z) ,\\ 
        \langle z\rangle^\ell\pp_t^j u^{b,2}_2\in L^\infty(0,T;H^{15-2j}_xH^3_z)\cap L^2(0,T;H^{15-2j}_xH^4_z),
    \end{gather*}
    for all $\ell\in\mathbb{N}$, $i=0,1,\cdots,5$ and $j=0,1\cdots,7$. 
\end{lemma}
\begin{proof}
    The proof is similar to that of Lemma $\ref{regularity of wb0ub11}$ by applying Proposition \ref{prop1}, Lemmas \ref{regularity of wb0ub11} and \ref{regularity of I1}. We omit it for brevity.
\end{proof} 
\begin{lemma}\label{regularity of I2}
    Under the assumptions of Theorem $\ref{thm1}$,  there exists a  unique solution  $(u^{I,2},w^{I,2})$ of problem $\eqref{I2}$  satisfying
    \begin{gather*}
         \pp_t^{5}u^{I,2}\in L^2_TL^2_{xy},~ \pp_t^i u^{I,2}\in L^\infty_TH^{9-2i}\cap L^2_TH^{10-2i},~~i=0,1\cdots,4,\\
         w^{I,2}\in L^\infty_TL^{9}_{xy},~ \pp_t^j w^{I,2}\in L^\infty_TL^{10-2j}_{xy},~~j=1,2,\cdots,5.
    \end{gather*}
\end{lemma} 
\begin{proof}
    Denoting \begin{align*}
        \hat{u}^{I,2}(x,y,t)&=u^{I,2}(x,y,t)+2\begin{pmatrix}
                \left[\varphi^2(y)+\varphi^\prime(y)\int^{y}_0\varphi(\xi)\dd \xi\right]  \int_0^{+\infty} w^{b,1}(x,s,t)\mathrm{d}s \\  \int_0^\infty \int_\tau^{+\infty} \partial_xw^{b,0}(x,s,t)\mathrm{d}s\mathrm{d}\tau -\varphi(y)\int^{y}_0\varphi(\xi)\dd \xi \int_0^{+\infty} w^{b,1}(x,s,t)\mathrm{d}s 
        \end{pmatrix}  \\
    &=:u^{I,2}(x,y,t)+S^{I,2}(x,y,t).
    \end{align*} 
    and using $\eqref{I1}$, we have
    \begin{equation*}
    \begin{cases}
        \pp_t  \hat{u}^{I,2}+\left( \hat{u}^{I,2}\cdot\nabla\right)u^{I,0}+\left( u^{I,0}\cdot\nabla\right)\hat{u}^{I,2}+\nabla \hat{p}^{I,2}-\zeta\D \hat{u}^{I,2}-2\zeta\nabla^\perp w^{I,2}=\hat{f}^{I,2},\\ 
        \pp_t  w^{I,2}+\left( \hat{u}^{I,2}\cdot\nabla\right)w^{I,0}+\left( u^{I,0}\cdot\nabla\right)w^{I,2}+4\zeta w^{I,2}-2\zeta\nabla^\perp\cdot \hat{u}^{I,2}=\hat{g}^{I,2},\\
        \di \hat{u}^{I,2}=0,\\
        \hat{u}^{I,2}(x,0,t)=0,~\left(\hat{u}^{I,2},w^{I,2}\right)(x,y,0)=0.
    \end{cases}
\end{equation*}
where
\begin{align*}
   &\hat{f}^{I,2}=\pp_tS^{I,2}+S^{I,2}\cdot\nabla u^{I,0}+u^{I,0}\cdot\nabla S^{I,2}-\D S^{I,2}+f^{I,2},\\
    &\hat{g}^{I,2}=S^{I,2}\cdot\nabla w^{I,0}-2\zeta\nabla^\perp\cdot S^{I,2}+g^{I,2}.
\end{align*}
Using Proposition \ref{prop1} and Lemmas \ref{regularity of wb0ub11}-\ref{regularity of wb1ub21}, a direct calculation yields that
\begin{gather*}
     \pp_t^4 \hat{f}^{I,2}\in L^2_TL^2_{xy}, ~ \pp^j_t\hat{f}^{I,2}\in L^{\infty}_TH^{8-2j}_{xy}\cap L^2_TH^{8-2j}_{xy}, ~j=0,\cdots,3,\\
    \pp^i_t\hat{g}^{I,2}\in L^{\infty}_TH^{9-2j}_{xy}\cap L^2_TH^{9-2j}_{xy}, ~i=0,\cdots,4.
\end{gather*}
Hence, using the proposition $\ref{regularity of I}$, we can finish the proof of Lemma. 
\end{proof}
\begin{remark}
    To guarantee that the problem \eqref{I2} satisfies the fourth order compatibility condition, we need to propose the following conditions on the initial data
        \begin{align}\label{compatibility condition of I2}
        \pp_t^\ell u^{I,2}(0) \Big|_{y=0}=0,~\ell=0,1,\cdots,4.
    \end{align}
    which can be represented by $(u^{I,0},w^{I,0})$ via \eqref{I0}, \eqref{I1} and \eqref{I2}. Here, we omit the details for brevity. One can refer to Remark 3.7 in \cite{Wang and Wen} for some related discuss.
\end{remark} 
Finally, for the higher order inner layer profiles, we have the following Lemma.
\begin{lemma}\label{regularity_of_wb2ub31}
   Under the assumptions of Theorem $\ref{thm1}$,  there exists a unique solution  $w^{b,2}$ of the problem  satisfying
    \begin{align*}
        \langle z\rangle^\ell\pp_t^j w^{b,2}\in L^\infty(0,T;H^{7-2j}_xH^1_z)\cap L^2(0,T;H^{7-2j}_xH^2_z) , 
    \end{align*}
    for all $\ell\in\mathbb{N}$ and $j=0,1,2,3$. Furthermore, using , we have 
    \begin{gather*} 
        \langle z\rangle^\ell\pp_t^i u^{b,3}_1\in L^\infty(0,T;H^{7-2i}_xH^2_z)\cap L^2(0,T;H^{7-2i}_xH^3_z) ,\\
        \langle z\rangle^\ell\pp_t^i u^{b,4}_2\in L^\infty(0,T;H^{6-2i}_xH^3_z)\cap L^2(0,T;H^{6-2i}_xH^4_z) ,\\
        \langle z\rangle^\ell\pp_t^j u^{b,3}_2\in L^\infty(0,T;H^{10-2j}_xH^3_z)\cap L^2(0,T;H^{10-2j}_xH^4_z) ,
    \end{gather*}
    for all $\ell\in\mathbb{N}$, $i=0,1,2,3$ and $j=0,1\cdots,5$.
\end{lemma}
\begin{proof}
    The proof is similar to that of Lemma $\ref{regularity of wb0ub11}$ by using  Proposition \ref{prop1}, Lemmas \ref{regularity of wb0ub11}-\ref{regularity of I2}. We omit it for brevity.
\end{proof}

\subsection{Construction of the approximate solution}
Based on the analysis in Section \ref{Asymptotic analysis}, we define
an approximate solution for the system $\eqref{uepsilon}$ as follows:
\begin{equation*} 
\begin{cases}  
    u^a(x,y,t)=u^I(x,y,t)+u^b\left(x,\frac{y}{\sqrt{\varepsilon}},t\right)+\e^\frac{3}{2}S(x,y,t),\\
    p^a(x,y,t)=p^I(x,y,t),\\
    w^a(x,y,t)=w^I(x,y,t)+w^b\left(x,\frac{y}{\sqrt{\varepsilon}},t\right),
\end{cases}  
\end{equation*}
where
\begin{gather*}
     u^I=u^{I,0}+\e^\frac{1}{2}u^{I,1}+\e u^{I,2},~~~p^I=p^{I,0}+\e^\frac{1}{2}p^{I,1} +\e p^{I,2},~~~w^I=w^{I,0}+\e^\frac{1}{2}w^{I,1}+\e w^{I,2},\\  u^b=\e^\frac{1}{2}u^{b,1}+\e u^{b,2}+\e^\frac{3}{2}u^{b,3}+\e^2(0,u_2^{b,4})^\top,~~~w^b=w^{b,0}+\e^\frac{1}{2}w^{b,1}+\e w^{b,2}.
\end{gather*}
and
\begin{align*}
    S(x,y,t)=&\left(\begin{matrix}
        -\varphi^\prime(y)\int_0^{+\infty}u^{b,2}_1(x,s,t)\dd s-\left[\varphi^2(y)+\varphi^\prime(y)\int^{y}_0\varphi(\xi)\dd \xi\right] u^{B,3}_1(x,0,t)\\
        -\varphi(y)\int_0^{+\infty}\pp_xu^{b,2}_1(x,s,t)\dd s+\varphi(y)\int^{y}_0\varphi(\xi)\dd \xi ~\pp_xu^{b,3}_1(x,0,t)
    \end{matrix}\right)\\ \notag
    &+\e^\frac{1}{2}\left(\begin{matrix}
       -\varphi^\prime(y)\int_0^{+\infty}u^{b,3}_1(x,s,t)\dd s\\
       -\varphi(y)\int_0^{+\infty}\pp_xu^{b,3}_1(x,s,t)\dd s
    \end{matrix}\right),
\end{align*}
 with $\varphi$ the cut-off function defined by \eqref{cut-off function}. Let $(u,p,w)$ be the solution of problem \eqref{uepsilon}-\eqref{uepsilon_ini_bou_cond}. Then, we define the error terms by 
\begin{equation}
    U^\e(x,y,t):=u-u^a,~~P^\e(x,y,t):=p-p^a,~~W^\e(x,y,t):=w-w^a.
\end{equation} 
Substituting $(u^a,p^a,w^a)$ into \eqref{uepsilon}-\eqref{uepsilon_ini_bou_cond}, with the help of the equations of inner and
outer layer profiles in Section  \ref{Asymptotic analysis}, we find that the error functions satisfy the following problem:
\begin{equation} \label{error equations}
\begin{cases}  
&\pp_t  U^\e+\left( U^\e \cdot \nabla \right)u^a+\left( u^a \cdot \nabla \right)U^\e+\left( U^\e \cdot \nabla \right)U^\e+\nabla P^\e-\left( \e+\zeta\right)\Delta U^\e+2\zeta \nabla^\perp W^\e=F,\\
&\pp_t  W^\e+\left( U^\e \cdot \nabla \right)w^a+\left( u^a \cdot \nabla \right)W^\e+\left( U^\e \cdot \nabla \right)W^\e+4\zeta W^\e-\e \Delta W^\e-2\zeta\nabla^\perp\cdot U^\e=G,\\
&\di U^\e=0,\\
&\left(U^\e,W^\e\right) \left(x,y,0\right)=0,~~~\left(U^\e,W^\e\right) (x,0,t)=0,
\end{cases}  
\end{equation} 
where
\begin{gather*}     
F=-\pp_t  u^a-\left( u^a \cdot \nabla \right)u^a-\nabla P^a+\left(\e+\zeta\right)\Delta u^a+2\zeta \nabla^\perp w^a,\\
G=-\pp_t  w^a-\left( u^a \cdot \nabla \right)w^a-4\zeta w^a+\e\Delta w^a+2\zeta\nabla^\perp\cdot u^a. 
\end{gather*} 
Moreover, using the equations of inner and outer layer profiles, we can split $F$ and $G$ as follows.
$$
F = \sum_{i=1}^8 F_i,~~~~G = \sum_{i=1}^{14}G_i.
$$
See $\eqref{source term F}$ and $\eqref{source term G}$ for the detailed expressions.

\section{Justification of the vanishing angular viscosity limit}\label{Proof of the theorem} 
In this section, we give the proof of Theorems \ref{thm1} and \ref{thm2}. All the constants $C$ used in the proof may depend on $\zeta, T$ and the bounds obtained in Section \ref{Regularity of boundary layer profiles}, but are independent of $\varepsilon$. Without loss of generality, we assume that $\varepsilon\in (0,1]$ in the rest of the paper. 
\subsection{Estimates of the source terms}
To begin with, we give some basic estimates of the approximate solutions which will be frequently used later. 
\begin{lemma}\label{the estimate of approximate solution}
    Under the assumptions of Theorem \ref{thm1},  there exists a constant $C$ independent of $\e$ such that
    \begin{gather*}
        \| (u^I,w^I,\pp_t  u^I,\pp_t  w^I, \nabla u^I,\nabla w^I, \nabla^2 u^I,\nabla^2 w^I,\pp_t\nabla u^I,\pp_t\nabla w^I)\|_{L^\infty_TL^\infty_{xy}}\le C,\\  
        \|( \pp_x u^b, \pp_zu^b, \pp_t\pp_x u^b, \pp_z\pp_tu^b, \pp_x^2u^b, \pp_z\pp_xu^b)\|_{L^\infty_TL^\infty_{xy}}\le C,\\
         \|( \pp_x w^b,\langle z\rangle\pp_zw^b, \pp_t\pp_x w^b,\langle z\rangle\pp_z\pp_tw^b, \pp_x^2w^b,\langle z\rangle\pp_z\pp_xw^b)\|_{L^\infty_TL^\infty_{xy}}\le C,\\
        \| (\nabla w^a,\pp_t\pp_x\nabla w^a)\|_{L^\infty_TL^2_{xy}}\le C\e^{-\frac{1}{4}},~~~\left\| \left( \nabla(w^a-w^{b,0}),\pp_t\pp_x\nabla(w^a-w^{b,0}) \right)\right\|_{L^\infty_TL^\infty_{xy}} \le C,
    \end{gather*}
    and
    \begin{gather*}
        \|S\|_{L^\infty_TH^2_{xy}}+\|\pp_tS\|_{L^\infty_TH^3_{xy}}+\|\pp_t^2S\|_{L^\infty_TH^1_{xy}}\le C.
    \end{gather*}
\end{lemma}
\begin{proof}
    The proof can be completed by directly using the Lemma \ref{regularity of wb0ub11}-\ref{regularity_of_wb2ub31}, H\"older and Sobolev inequalities with the expressions of  the approximate solutions. We
omit the details for brevity.
\end{proof}
Next, for the source term $F$, we have the following estimate.
\begin{lemma}\label{estimate of F}
Under the assumptions of Theorem \ref{thm1},  then here exists a constant $C$ independent of $\e$ such that
\begin{align}
        \left\|  \left(F,\pp_x F,\pp_t F,\pp_t\pp_x F\right) \right\|_{L^\infty_TL^2_{xy}}\le C \e^\frac{5}{4}.
\end{align}
\end{lemma}

\begin{proof}
Thanks to \eqref{source term F},
we can estimate $F$ term by term. For $F_1$, using Lemmas \ref{regularity of wb1ub21}, \ref{regularity_of_wb2ub31} and \ref{the estimate of approximate solution}, we have 
    \begin{align*} 
       \notag \left\|  F_1 \right\|_{L^\infty_TL^2_{xy}}\le&\e \left\|   \pp_tu^{b,2}\right\|_{L^\infty_TL^2_{xy}}+\e^\frac{3}{2}\left\|  \pp_tu^{b,3} \right\|_{L^\infty_TL^2_{xy}}+\e^\frac{3}{2}\left\|  \pp_tS \right\|_{L^\infty_TL^2_{xy}}+\e^2\left\|  \pp_tu_2^{b,4} \right\|_{L^\infty_TL^2_{xy}}\\
       \notag  \le&\e^\frac{5}{4}\left\|   \pp_tu^{b,2}\right\|_{L^\infty_TL^2_xL^2_z}+\e^\frac{7}{4}\left\|   \pp_tu^{b,3}\right\|_{L^\infty_TL^2_{x}L^2_z}+\e^\frac{3}{2}\left\|  \pp_tS \right\|_{L^\infty_TL^2_{xy}}+\e^\frac{9}{4}\left\|  \pp_tu^{b,4}_2 \right\|_{L^\infty_TL^2_{x}L^2_z}\\
        \le&C\e^\frac{5}{4}.
    \end{align*}
    For $F_2$, we use the Taylor's  formula,
    \begin{align*}
        \notag \left\|  F_2 \right\|_{L^\infty_TL^2_{xy}}\le& \left\|   \frac{u^{I,0}_1(x,y,t)-u^{I,0}_1(x,0,t)}{y}\e^\frac{1}{2}z\e^\frac{1}{2}\pp_xu^{b,1}\right\|_{L^\infty_TL^2_{xy}}\\
        \notag &+\left\|   \frac{u^{I,0}_2(x,y,t)-u^{I,0}_2(x,0,t)-\pp_yu^{I,0}_2(x,0,t)y}{\frac{1}{2}y^2}\frac{1}{2}z^2\e\pp_zu^{b,1}\right\|_{L^\infty_TL^2_{xy}}\\
        \notag \le& C\e^\frac{5}{4} \left\|  u^{I,0}_1 \right\|_{L^\infty_TH^3_{xy}}\left\|  \langle z\rangle u^{b,1} \right\|_{L^\infty_TH^1_xL_{z}^2}+C\e^\frac{5}{4} \left\|  u^{I,0}_2 \right\|_{L^\infty_TH^4_{xy}}\left\|  \langle z\rangle^2 u^{b,1} \right\|_{L^\infty_TL^2_xH_{z}^1}\\
        \le&C\e^\frac{5}{4}.
    \end{align*}
A similar treatment yields that 
\begin{align*}
    \left\|  F_3 \right\|_{L^\infty_TL^2_{xy}}\le & \e^\frac{5}{4}\left\|   u^{I, 0}_1\right\|_{L^\infty_T L^{\infty}_{xy}}\left\|  u^{b,2} \right\|_{L^\infty_TH^1_xL_{z}^2}+\e^\frac{5}{4}\left\|   u^{I, 0}_2\right\|_{L^\infty_T L^{\infty}_{xy}}\left\|  u^{b,2} \right\|_{L^\infty_TL_{x}^2H^1_z}\\
    &+\e^\frac{5}{4}\left\|   u^{I, 0}\right\|_{L^\infty_T L^{\infty}_{xy}}\Big(
\left\|  u^{b,3} \right\|_{L^\infty_TH^1_xL^2_z}+\left\|  u^{b,3} \right\|_{L^\infty_TL^2_xH^1_z}+\e^\frac{1}{4}\left\| S  \right\|_{L^\infty_TL^2_{xy}}\\
&~~~~+\e^\frac{1}{2}\left\|  u^{b,4}_2 \right\|_{L^\infty_TH^1_xL^2_z}+\e^\frac{1}{2}\left\|  u^{b,4}_2 \right\|_{L^\infty_TL^2_xH^1_z}\Big)\\
\le&C\e^\frac{5}{4}.
\end{align*}
Using Proposition \ref{prop1}, Lemmas \ref{regularity of wb0ub11}, \ref{regularity of I1}, \ref{regularity of wb1ub21}, \ref{regularity of I2}, \ref{regularity_of_wb2ub31} and \ref{the estimate of approximate solution}, we have 
\begin{align*}
  \left\|  F_4 \right\|_{L^\infty_TL^2_{xy}}\le&\e^\frac{5}{4} \left\| u^{I,1} \right\|_{L^\infty_TL^{\infty}_{xy}} \Big( \left\|   u^{I,2}\right\|_{L^\infty_TH^1_{xy}}+\left\|u^{b,2}   \right\|_{L^\infty_TH^1_xL_{z}^2}+\left\|u^{b,2}   \right\|_{L^\infty_TL_{x}^2H^1_z}+\e^\frac{1}{2}\left\|u^{b,3}   \right\|_{L^\infty_TH^1_xL_{z}^2}\\
  &+\e^\frac{1}{2}\left\|u^{b,3}   \right\|_{L^\infty_TL_{x}^2H^1_z}+\left\| S  \right\|_{L^\infty_TL^2_{xy}}+\e\left\|u^{b,4}_2   \right\|_{L^\infty_TH^1_xL_{z}^2}+\e\left\|u^{b,4}_2   \right\|_{L^\infty_TL_{x}^2H^1_z}\Big)\\
  \le&C\e^\frac{5}{4}.
\end{align*}
Using the defination of $u^a$, it is easy to get $$\left\|\pp_x(u^a-u^{I,0}) \right\|_{L^\infty_TL^{\infty}_{xy}}\le C\e^\frac{1}{2}.$$
Hence, for $F_5$, we have
\begin{align*}
    \left\|  F_5 \right\|_{L^\infty_TL^2_{xy}}\le&\e^\frac{5}{4}\left\|  u^{I,0} \right\|_{L^\infty_TH^2_{xy}}\left\| \langle z\rangle u^{b,1}  \right\|_{L^\infty_TL^2_xL^2_z}+\e^\frac{3}{4}\left\|  \pp_x(u^a-u^{I,0}) \right\|_{L^\infty_TL^{\infty}_{xy}}\left\| u^{b,1}  \right\|_{L^\infty_TL^2_xL^2_z}\\
    \le&C\e^\frac{5}{4}.
\end{align*}
A similar argument yields that
\begin{align*}
    \left\|  F_6 \right\|_{L^\infty_TL^2_{xy}}\le&\e^\frac{5}{4}\left\| \nabla u^a  \right\|_{L^\infty_TL^{\infty}_{xy}}\Big(\left\|u^{b,2}   \right\|_{L^\infty_TL^2_xL^2_z}+\e^\frac{1}{2}\left\|u^{b,3}   \right\|_{L^\infty_TL^2_xL^2_z}\\
    &+\e^\frac{1}{4}\left\| S  \right\|_{L^\infty_TL^2_{xy}}+\e\left\|u^{b,4}_2   \right\|_{L^\infty_TL^2_xL^2_z}\Big)\\
    \le&C\e^\frac{5}{4}.
\end{align*}
A similar treatment yields that, 
\begin{align*}
    \left\|  F_7 \right\|_{L^\infty_TL^2_{xy}}\le&\e^\frac{5}{4}\Big(\e^\frac{1}{4}
\left\| u^{I,1}  \right\|_{L^\infty_TH^2_{xy}}+\e^\frac{1}{2}\left\| u^{b,1}  \right\|_{L^\infty_TH^2_xL_{z}^2}+\e^\frac{3}{4}\left\| u^{I,2}  \right\|_{L^\infty_TH^2_{xy}} \\
&+\e\left\| u^{b,2}  \right\|_{L^\infty_TH^2_xL_{z}^2}+\e^\frac{3}{2}\left\| u^{b,3}  \right\|_{L^\infty_TH^2_xL_{z}^2}+\e^\frac{1}{2}\left\| u^{b,3}  \right\|_{L^\infty_TL_{x}^2H^2_z} \\
&+\e^\frac{5}{4}\left\|   S\right\|_{L^\infty_TH^2_{xy}}+\e^2\left\| u^{b,4}_2  \right\|_{L^\infty_TH^2_xL_{z}^2}+\e\left\| u^{b,4}_2  \right\|_{L^\infty_TL_{x}^2H^2_z} \Big)\\
\le&C\e^\frac{5}{4}.
\end{align*}
Finally, we can estimate $F_8$ as follows.
\begin{align*}
     \left\|  F_8 \right\|_{L^\infty_TL^2_{xy}}\le&\zeta\e^\frac{5}{4}\Big(\left\| u^{b,2}  \right\|_{L^\infty_TH^2_xL_{z}^2}+\e^\frac{1}{2}\left\| u^{b,3}  \right\|_{L^\infty_TH^2_xL_{z}^2}+\e^\frac{1}{4}\left\|   S\right\|_{L^\infty_TH^2_{xy}}\\
     &+\e\left\| u^{b,4}_2  \right\|_{L^\infty_TH^2_xL_{z}^2}+\left\| u^{b,4}_2  \right\|_{L^\infty_TL_{x}^2H^2_z}\Big)+\zeta\e^\frac{5}{4}\left\|   w^{b,2}\right\|_{L^\infty_TH^1_xL_{z}^2}\\
     \le&C\e^\frac{5}{4}.
\end{align*}
Combining the estimates of $F_1,\cdots,F_8$, we have
\begin{align*}
    \left\| F  \right\|_{L^\infty_TL^2_{xy}}=\sum_{i=1}^{8}\left\| F_i  \right\|_{L^\infty_TL^2_{xy}}\le C\e^\frac{5}{4}.
\end{align*}
Noticing that $\pp_x F, \pp_t F, \pp_t\pp_x F$ only involving the time 
derivative $\partial_t$ and the tangential derivative $\partial_x$, we can prove the estimates of $\pp_x F, \pp_t F, \pp_t\pp_x F$ in a similar way. 
Here, we omit the details for brevity. The proof is complete.
\end{proof}
Finally, for the source term $G$, we have the following estimates.
\begin{lemma}\label{estimate of G}  
Under the assumptions of Theorem \ref{thm1},  there exists a constant $C$ independent of $\e$ such that
\begin{align}
        \left\|  \left(G,\pp_x G,\pp_tG,\pp_t\pp_x G\right) \right\|_{L^\infty_TL^2_{xy}}\le C \e^\frac{5}{4}.
\end{align}
\end{lemma}

\begin{proof}
    For $G_1$ and $G_2$, similar to the estimate of $F_2$, we have
    \begin{align*}
        \left\|  (G_1 ,  G_2) \right\|_{L^\infty_TL^2_{xy}}\le& \e^\frac{5}{4}
\left\|  \pp_y^2u^{I,0} \right\|_{L^\infty_TL^{\infty}_{xy}}\left\| \langle z \rangle ^2w^{b,0}  \right\|_{L^\infty_TH^1_xL_{z}^2}+\e^\frac{5}{4}
\left\|  \pp_y^3u^{I,0} \right\|_{L^\infty_TL^{\infty}_{xy}}\left\| \langle z \rangle ^3w^{b,0}  \right\|_{L^\infty_TL_{x}^2H^1_z}\\
        \le& C\e^\frac{5}{4}.
    \end{align*}
Based on the regularity of boundary profiles, we can give 
    \begin{align*}
        \left\|  G_3 \right\|_{L^\infty_TL^2_{xy}}\le&\e^\frac{5}{4}\left\|   \pp_y u^{I,0}\right\|_{L^\infty_TL^{\infty}_{xy}}\left\| 
\langle z \rangle w^{b,1}  \right\|_{L^\infty_TH^1_xL_{z}^2}+\e^\frac{5}{4}\left\|   \pp_y^2 u^{I,0}\right\|_{L^\infty_TL^{\infty}_{xy}}\left\| 
\langle z \rangle^2 w^{b,1}  \right\|_{L^\infty_TL_{x}^2H^1_z}\\
        \le& C\e^\frac{5}{4}.
    \end{align*}
A similar argument yields that
    \begin{align*}
        \left\|  G_4 \right\|_{L^\infty_TL^2_{xy}}\le&\e^\frac{5}{4}\left\|   u^{I,0}\right\|_{L^\infty_TL^{\infty}_{xy}}\left\|  w^{b,2} \right\|_{L^\infty_TH^1_xL_{z}^2}+\e^\frac{5}{4}\left\|   \pp_y u^{I,0} \right\|_{L^\infty_TL^{\infty}_{xy}}\left\| \langle z \rangle w^{b,2}  \right\|_{L^\infty_TL^2_xH_{z}^1}\\
        &+\e^\frac{5}{4}\left\|   \pp_y u^{I,1} \right\|_{L^\infty_TL^{\infty}_{xy}}\left\| \langle z \rangle w^{b,0}  \right\|_{L^\infty_TH_{x}^1L^2_z}\\
        \le& C\e^\frac{5}{4}.
    \end{align*}
A similar treatment yields that
    \begin{align*}
        \left\|  G_5 \right\|_{L^\infty_TL^2_{xy}}\le&\e^\frac{5}{4}\left\|   \pp_y^2u^{I,1}\right\|_{L^\infty_TL^{\infty}_{xy}}\left\|   \langle z \rangle ^2w^{b,0}\right\|_{L^\infty_TL^2_xH_{z}^1}+\e^\frac{5}{4}\left\|   u^{I,1}\right\|_{L^\infty_TL^{\infty}_{xy}}\left\|  w^{b,1} \right\|_{L^\infty_TH^1_xL_{z}^2}\\
        \le& C\e^\frac{5}{4}.
    \end{align*}
    Moreover, for $G_6$, we have
    \begin{align*}
        \left\|  G_6 \right\|_{L^\infty_TL^2_{xy}}\le&\e^\frac{5}{4}
\left\|  \pp_yu^{I,1 }  \right\|_{L^\infty_TL^{\infty}_{xy}}
\left\| 
\langle z \rangle w^{b,1}  \right\|_{L^\infty_TL^2_xH_{z}^1}+
\e^\frac{7}{4}  \left\|  u^{I,1 }  \right\|_{L^\infty_TL^{\infty}_{xy}}
\left\| 
 w^{b,2}  \right\|_{L^\infty_TH_{x}^1L^2_z}\\
 &+\e^\frac{5}{4}\left\|  u^{I,1 }  \right\|_{L^\infty_TL^{\infty}_{xy}}
\left\| 
 w^{b,2}  \right\|_{L^\infty_TL^2_xH_{z}^1}\\
 \le&C\e^\frac{5}{4}.
    \end{align*}
A similar argument yields that
    \begin{align*}
        \left\|  G_7 \right\|_{L^\infty_TL^2_{xy}}\le&\e^\frac{5}{4}\left\|   \pp_y\pp_xw^{I,0}\right\|_{L^\infty_TL^{\infty}_{xy}}\left\|  \langle z \rangle u^{b,1} \right\|_{L^\infty_TL^2_xL^2_z}+\e^\frac{7}{4}\left\| u^{b,1}  \right\|_{L^\infty_TL^2_xL^2_z}\left\|  \pp_xw^{I,2} \right\|_{L^\infty_TL^{\infty}_{xy}}\\
        &+\e^\frac{5}{4}\left\|   u^{b,1}\right\|_{L^\infty_TL^{\infty}_{x}L^\infty_z}\left\|   w^{b,2}\right\|_{L^\infty_TH^1_xL_{z}^2}\\
        \le&C\e^\frac{5}{4}.
    \end{align*}
Next, using the regularity of boundary profiles, we have
    \begin{align*}
        \left\|  G_8 \right\|_{L^\infty_TL^2_{xy}}\le&\e^\frac{5}{4}\left\|  u^{I,2} \right\|_{L^\infty_TL^{\infty}_{xy}}\left\| w^{b,0}  \right\|_{L^\infty_TH^1_xL_{z}^2}+\e^\frac{5}{4}\left\|  \pp_y^2u^{I,2} \right\|_{L^\infty_TL^{\infty}_{xy}}\left\| \langle z \rangle w^{b,0}  \right\|_{L^\infty_TL_{x}^2H^1_z}\\
        &+\e^\frac{7}{4}\left\|   u^{I,2}\right\|_{L^\infty_TL^{\infty}_{xy}}\left\| w^{b,1}  \right\|_{L^\infty_TH^1_xL_{z}^2}\\
        \le&C\e^\frac{5}{4}.
    \end{align*}
A similar argument yields that
    \begin{align*}
        \left\|  G_9 \right\|_{L^\infty_TL^2_{xy}}\le&\e^\frac{5}{4}\left\|  u^{I,2} \right\|_{L^\infty_TL^{\infty}_{xy}}\left\|  w^{b,1} \right\|_{L^\infty_TL^2_xH_{z}^1}+\e^\frac{3}{2}\left\|  u^{I,2} \right\|_{L^\infty_TL^{\infty}_{xy}}\left( \left\|  \nabla w^{I,1} \right\|_{L^\infty_TL^2_{xy}}+\e^\frac{1}{2}\left\|  \nabla w^{I,2} \right\|_{L^\infty_TL^2_{xy}}\right)\\
        &+\e^\frac{9}{4}\left\| u^{I,2}  \right\|_{L^\infty_TL^{\infty}_{xy}}\left\|  w^{b,2} \right\|_{L^\infty_TH_{x}^1L^2_z}+\e^\frac{7}{4}\left\| u^{I,2}  \right\|_{L^\infty_TL^{\infty}_{xy}}\left\|  w^{b,2} \right\|_{L^\infty_TL^2_xH_{z}^1}\\
        \le&C\e^\frac{5}{4}.
    \end{align*}
A similar treatment yields that 
    \begin{align*}
        \left\|  G_{10} \right\|_{L^\infty_TL^2_{xy}}\le&\e^\frac{5}{4}\left\|  u^{b,2} \right\|_{L^\infty_TL^2_xL^2_z}\left(\left\|  \pp_xw^{I,0} \right\|_{L^\infty_TL^{\infty}_{xy}}+\left\|  \pp_yw^{I,0} \right\|_{L^\infty_TL^{\infty}_{xy}}\right)\\
        \le&C\e^\frac{5}{4}.
    \end{align*}
    For $G_{11}$, similarly, we have
    \begin{align*}
        \left\|  G_{11} \right\|_{L^\infty_TL^2_{xy}}\le&\e^\frac{7}{4}\left\|   u^{b,2}\right\|_{L^\infty_TL^\infty_xL^\infty_z}\left\| w^{b,1}  \right\|_{L^\infty_TH^1_xL_{z}^2}+\e^\frac{7}{4}\left\|u^{b,2}\right\|_{L^\infty_TL^2_xL^2_z}\Big(\left\| \nabla w^{I,1}  \right\|_{L^\infty_TL^{\infty}_{xy}}\\
        &+\e^\frac{1}{2}\left\| \nabla w^{I,2}  \right\|_{L^\infty_TL^{\infty}_{xy}}\Big)+\e^\frac{9}{4}\left\|   u^{b,2}\right\|_{L^\infty_TL^\infty_xL^\infty_z}\left\| w^{b,2}  \right\|_{L^\infty_TH^1_xL_{z}^2}\\
        &+\e^\frac{7}{4}\left\|   u^{b,2}\right\|_{L^\infty_TL^\infty_xL^\infty_z}\left\| w^{b,2}  \right\|_{L^\infty_TL_{x}^2H^1_z}+\e^\frac{7}{4}\left\|  u^{b,3} \right\|_{L^\infty_TL^2_xL^2_z}
\left\|   \pp_x w^{b,0}\right\|_{L^\infty_TL^\infty_xL^\infty_z}\\
        \le&C\e^\frac{7}{4}.
    \end{align*}
    For $G_{12}$, using the Lemma \ref{the estimate of approximate solution}, we have
    \begin{align*}
        \left\|  G_{12} \right\|_{L^\infty_TL^2_{xy}}\le&\e^\frac{7}{4}\left\| u^{b,3}  \right\|_{L^\infty_TL^2_xL^2_z}\left\|\nabla (w^a-w^{b,0})   \right\|_{L^\infty_TL^\infty_{xy}}+\e^\frac{3}{2}\left\|  S \right\|_{L^\infty_TL^{\infty}_{xy}}\left\|  \nabla w^a \right\|_{L^\infty_TL^2_{xy}}\\
        &+\e^2\left\| u^{b,4}_2  \right\|_{L^\infty_TL^\infty_xL^\infty_z}\left\|  \nabla w^a \right\|_{L^\infty_TL^2_{xy}}+\e^\frac{5}{4}\left\| u^{b,3}_2  \right\|_{L^\infty_TL^\infty_xL^\infty_z}\left\|   w^{b,0}\right\|_{L^\infty_TL^2_xH_{z}^1}\\
        \le&C\e^\frac{3}{2}.
    \end{align*}
A similar argument yields that
    \begin{align*}
        \left\|  G_{13} \right\|_{L^\infty_TL^2_{xy}}\le& \e^\frac{3}{2}\Big( \left\|  w^{I,1}\right\|_{L^\infty_TH^2_{xy}}+\e^\frac{1}{4}
\left\|  w^{b,1} \right\|_{L^\infty_TH^2_xL_{z}^2}+\e^\frac{1}{2}  \left\|  w^{I,2}\right\|_{L^\infty_TH^2_{xy}}+\e^\frac{3}{4}
\left\|  w^{b,2} \right\|_{L^\infty_TH^2_xL_{z}^2}
\Big)\\
        \le&C\e^\frac{3}{2}.
    \end{align*}
A similar treatment yields that
    \begin{align*}
        \left\|  G_{14} \right\|_{L^\infty_TL^2_{xy}}\le&2\zeta\e^\frac{3}{2}\left(\e^\frac{1}{4}\left\|  u^{b,3}_2 \right\|_{L^\infty_TH^1_xL_{z}^2}+\left\|  \nabla S\right\|_{L^\infty_TL^2_{xy}}+\e^\frac{1}{4}\left\| u^{b,4}_2  \right\|_{L^\infty_TL^2_xH_{z}^1}\right)\\
        \le&C\e^\frac{3}{2}.
    \end{align*}
Combining $G_i$, we have
\begin{align*}
    \left\| G  \right\|_{L^\infty_TL^2_{xy}}=\sum_{i=1}^{14}\left\| G_i  \right\|_{L^\infty_TL^2_{xy}}\le C\e^\frac{5}{4}.
\end{align*}
The estimates of  $\pp_x G, \pp_t G, \pp_t\pp_x G$ can be deduced in a similar way.    
 We  omit the details for  brevity.
\end{proof}
\subsection{Estimates for the error terms}\label{Estimates_error} 
To begin with, we have the following $L^\infty_TL^2_{xy}$ estimate of $(U^\varepsilon,W^\varepsilon)$.
\begin{lemma}\label{lem_L_infty_L_2}
Under the assumptions of Theorem \ref{thm1}, for any $0 < \e < 1$, there exists a constant $C$ independent of $\e$, such that
\begin{align*}
    \left\| \left(U^\e,W^\e\right)\right\|^2+\left(\e+\zeta \right)\int_0^T\|\nabla U^\e\|^2\dd t+\e\int_0^T\|\nabla W^\e\|^2\dd t\le C\e^\frac{5}{2}.
\end{align*}
\end{lemma}

\begin{proof}
Multiplying $\eqref{error equations}_1$ by $U^\e$ and integrating the result by parts, we get

\begin{align}\label{eq_U_epsilon_L2}
       \notag \frac{1}{2} \frac{\dd }{\dd t}\| U^\e\|^2+(\e+\zeta)\|\nabla U^\e\|^2
       &=\langle u^a\otimes U^\e,\nabla U^\e\rangle-2\zeta\langle \nabla^\perp W^\e, U^\e \rangle+\langle F,U^\e \rangle \\
       &\leq \frac{1}{8}\zeta \| \nabla U^\e \|^2+C\left( \| u^a\|^2_{L^\infty}+1\right) \left\| (U^\e,W^\e) \right\|^2+\|F\|^2.
\end{align} 
Multiplying $\eqref{error equations}_2$ by $W^\e$ and integrating the result by parts, we get
\begin{align}\label{eq_W_epsilon_L2}
   \notag \frac{1}{2} \frac{\dd}{\dd t}\|W^\e\|^2&+\e\|\nabla W^\e\|^2+4\zeta\|W^\e\|^2\\
    &=-\langle U^\e\cdot\nabla  w^a,W^\e\rangle+2\zeta\langle\nabla^\perp\cdot U^\e,W^\e\rangle+\langle G,W^\e\rangle.
\end{align}
Using the Hardy inequality, we handle the first term on the right hand side of \eqref{eq_W_epsilon_L2} as follows.
\begin{align}\label{U_L2_WB}
     \notag -\langle U^\e\cdot\nabla w^a,W^\e\rangle&=-\langle U^\e\cdot(\nabla w^I+\pp_xw^b),W^\e\rangle - \langle U^\e_2\pp_y w^b,W^\e\rangle\\ \notag
    &=-\langle U^\e\cdot(\nabla w^I+\pp_xw^b),W^\e\rangle - \left\langle \frac{1}{\sqrt{\e}}U^\e_2\pp_z w^b,W^\e\right\rangle\\ \notag
    &=-\langle U^\e\cdot(\nabla w^I+\pp_xw^b),W^\e\rangle - \left\langle \frac{1}{y}U^\e_2z\pp_z w^b,W^\e\right\rangle\\ 
    &\leq\|U^\e\|\left\|\left(\nabla w^I,\pp_x w^b\right)\right\|_{L^\infty}\|W^\e\|+\|\pp_y U^\e\|\|\langle z\rangle \pp_z w^b\|_{L^\infty}\|W^\e\|.
\end{align}
Hence, substituting \eqref{U_L2_WB} into \eqref{eq_W_epsilon_L2}, we have
\begin{align}\label{eq_W_epsilon_L2_1}
    \notag \frac{1}{2} \frac{\dd}{\dd t}\|W^\e\|^2&+\e\|\nabla W^\e\|^2+2\e\| \di W^\e\|^2+4\zeta\|W^\e\|^2\\ 
    &\leq \frac{1}{4}\zeta\|\nabla U^\e\|^2+C\left( \| (\nabla w^I, \pp_x w^b, \langle z\rangle \pp_z w^b)\|^2_{L^\infty}+1\right) \left\| (U^\e,W^\e) \right\|^2+\|G\|^2.
\end{align}
Suming \eqref{eq_U_epsilon_L2} and \eqref{eq_W_epsilon_L2_1}, we get
\begin{align*}
     \frac{\dd}{\dd t}&\left(\left\| \left(U^\e,W^\e\right)\right\|^2\right)+\left(\e+\zeta \right)\|\nabla U^\e\|^2+\e\|\nabla W^\e\|^2 \\
    &\le  \|F\|^2+\|G\|^2+C\left( \left\| \left( u^a,\nabla w^I,\pp_x w^b,\langle z\rangle\pp_z w^b\right)\right\|^2_{L^{\infty}}+1\right)\left\| \left(U^\e,W^\e\right)\right\|^2,
\end{align*}
which, together with  Gronwall inequality, Lemmas \ref{the estimate of approximate solution}-\ref{estimate of G}, implies
\begin{align*}
    \left\| \left(U^\e,W^\e\right)\right\|^2+\left(\e+\zeta \right)\int_0^T\|\nabla U^\e\|^2\dd t+\e\int_0^T\|\nabla W^\e\|^2\dd t\le C_1e^{C_1T}\e^\frac{5}{2}.
\end{align*}
The proof is complete.
\end{proof}
Next, we have the following $L^\infty_TL^2_{xy}$ estimate of $(\partial_xU^\varepsilon,\partial_xW^\varepsilon)$.
\begin{lemma}
Under the assumptions of Theorem \ref{thm1}, for any  $0 < \e < 1$, there exists a constant $C$ independent of $\e$, such that
    \begin{align*}
    \left\| \left(\pp_xU^\e,\pp_xW^\e\right)\right\|^2+\zeta \int_0^T\|\nabla \pp_xU^\e\|^2\dd t+\e\int_0^T\|\nabla \pp_xW^\e\|^2\dd t\le C \e^\frac{3}{2}.
\end{align*}
\end{lemma}

\begin{proof} 
Multiplying $\pp_x\eqref{error equations}_1$ by $\pp_xU^\e$ and integrating by parts, we get
\begin{align}\label{eq_U_x_L2}
    \notag \frac{1}{2}\frac{\dd}{\dd t}\left\|\pp_xU^\e \right\|^2+\left(\e+\zeta \right)\left\|\nabla\pp_xU^\e\right\|^2=&~\left\langle u^a \otimes \pp_xU^\e,\nabla \pp_xU^\e\right\rangle+\left\langle U^\e\otimes \pp_xu^a,\nabla \pp_xU^\e\right\rangle\\ \notag
    &+\left\langle\pp_x u^a \otimes U^\e,\nabla \pp_xU^\e\right\rangle-\left\langle\pp_xU^\e\cdot \nabla U^\e, \pp_xU^\e\right\rangle\\ 
    &+2\zeta\left\langle \nabla^\perp\pp_xW^\e,\pp_xU^\e\right\rangle+\left\langle\pp_xF,\pp_xU^\e \right\rangle \notag \\
    &=\sum_{i=1}^{6}I_i.
\end{align}
For $I_1$, we have
\begin{align*}
    I_1 \le \left\|u^a \right\|_{L^\infty}\left\|\pp_xU^\e \right\|\left\| \nabla\pp_xU^\e\right\|\le C \left\| u^a\right\|^2_{L^\infty}\left\| \pp_xU\right\|^2+\frac{1}{16}\zeta\left\| \nabla\pp_xU\right\|^2.
\end{align*}
For $I_2$ and $I_3$, we have
\begin{align*}
   I_2=\notag \left\langle U^\e\otimes \pp_xu^a,\nabla \pp_xU^\e\right\rangle
    \le \left\|U^\e \right\|\left\|\pp_xu^a \right\|_{L^\infty}\left\| \nabla\pp_x U^\e\right\|
    \le C\left\|\pp_xu^a \right\|^2_{L^\infty}\left\|U \right\|^2+\frac{1}{16}\zeta\left\| \nabla\pp_xU\right\|^2.
\end{align*}
and
\begin{align*}
   I_3  \le C\left\| \pp_x u^a \right\|^2_{L^\infty}\left\|U \right\|^2+\frac{1}{16}\zeta\left\| \nabla\pp_xU\right\|^2.
\end{align*}
For $I_4$ using the Ladyzhenskaya inequality, we have
\begin{align*}
   \notag I_4= -\left\langle\pp_xU^\e\cdot \nabla U^\e, \pp_xU^\e\right\rangle
    &   \le\|\pp_xU^\e\|_{L^4}\|\pp_xU^\e \|_{L^4}\|\nabla U^\e\|\\
   \notag  &\le\|\pp_xU^\e\|\|\nabla\pp_xU^\e\| \|\nabla U^\e\|\\
    &\le C\|\nabla U^\e\|^2\|\pp_xU^\e\|^2+\frac{1}{16}\zeta\|\nabla\pp_xU^\e\|^2.
\end{align*}
For the last two terms, we have
\begin{align*}
 \notag
   I_5+I_6&\le2\zeta\left\| \pp_xW^\e\right\|\left\| \nabla\pp_xU^\e\right\|+\left\| \pp_xF\right\|\left\| \pp_xU^\e\right\|\\
   &\le C\|\pp_xW^\e\|^2+C\|U^\e\|^2+\frac{1}{16}\zeta\|\nabla\pp_xU\|^2+\|\pp_xF\|^2.
\end{align*}
Substituting the estimates of $I_1,\cdots,I_6$ into \eqref{eq_U_x_L2}, we have 
\begin{align}\label{eq_U_x_L2_1}
    \notag \frac{1}{2}\frac{\dd}{\dd t}\left\|\pp_xU^\e \right\|^2+ \frac{11}{16}\zeta \left\|\nabla\pp_xU^\e\right\|^2 \leq&\, C\left( \left\| u^a\right\|^2_{L^\infty}\left\| \pp_xU\right\|^2 + (\left\| \pp_x u^a \right\|^2_{L^\infty}+1)\left\|U \right\|^2\right) \\
    & + C\left(1+\|\nabla U^\e\|^2 \right)\left\| (\pp_xW^\e,\pp_xU^\e )\right\|^2 + \|\pp_xF\|^2.
\end{align}
Next, we deal with the estimate of $\partial_xW^\varepsilon$. Taking $L^2$ inner product of $\pp_x\eqref{error equations}_2 $ with $\pp_xW^\e$, and using integration by parts,
we have
\begin{align}\label{eq_W_x_L2}
    \notag \frac{1}{2}\frac{\dd}{\dd t}\left\|\pp_xW^\e \right\|^2+&\e\|\nabla\pp_xW^\e\|^2+4\zeta\|\pp_xW^\e\|\\ \notag
    =&~-\left\langle \pp_xU^\e\cdot \nabla w^a,\pp_xW^\e\right\rangle-\left\langle U^\e\cdot \nabla \pp_xw^a,\pp_xW^\e\right\rangle\\ \notag
    &~-\left\langle \pp_xu^a\cdot\nabla W^\e,\pp_xW^\e\right\rangle -\left\langle \pp_xU^\e\cdot\nabla W^\e,\pp_xW^\e\right\rangle\\
    &~+2\zeta\left\langle \nabla\times\pp_xU^\e,\pp_xW^\e\right\rangle+\left\langle\pp_xG,\pp_xW^\e \right\rangle\notag\\ 
    =&\sum_{i=1}^{6}J_i.
\end{align}
Similar to \eqref{U_L2_WB}, for $J_1$, we have
\begin{align*}
    \notag J_1&\leq\|\pp_xU^\e\|\|(\pp_x w^b,\nabla w^I)\|_{L^\infty}\|\pp_xW^\e\|+\|\pp_y \pp_xU^\e\|\|\langle z\rangle \pp_z w^b\|_{L^\infty}\|\pp_xW^\e\|\\
    &\le \left\|\left(\nabla w^I,\pp_xw^b,\langle z\rangle\pp_z w^b\right) \right\|^2_{L^\infty}\left\| (\pp_xW^\e,\pp_xU^\e )\right\|^2+\frac{1}{16}\zeta\|\nabla\pp_xU^\e\|^2.
\end{align*}
For $J_2$, we have
\begin{align*}
     J_2=\left\langle\partial_x w^a   U^\e,\nabla\pp_xW^\e \right\rangle\le\|\pp_xw^a\|_{L^\infty}\|U^\e\|\|\nabla\pp_xW^\e\|\le C\e^{-1}\|U^\e\|^2\|\pp_xw^a\|_{L^\infty}^2+\frac{1}{16}\e\|\nabla\pp_xW^\e\|^2.
\end{align*}
Similarly,
\begin{align*}
      J_3  \le\|\pp_x u^a\|_{L^\infty}\|W^\e\|\|\nabla\pp_xW^\e\|\le C\e^{-1}\|W^\e\|^2\|\pp_x u^a\|_{L^\infty}^2+\frac{1}{16}\e\|\nabla\pp_xW^\e\|^2 .
\end{align*}
For $J_4$, using the Ladyzhenskaya inequality, we have
\begin{align*}
     \notag  J_4  &=\left\langle W^\e \pp_xU^\e ,\nabla\pp_xW^\e\right\rangle\le\|\pp_xU^\e\|_{L^4}\| W^\e\|_{L^4}\|\nabla\pp_xW^\e\|\\ \notag 
     &\le\|\pp_xU^\e\|^\frac{1}{2}\|\nabla\pp_xU^\e\|^\frac{1}{2}\|W^\e\|^\frac{1}{2}\|\nabla W^\e\|^\frac{1}{2}\|\nabla\pp_xW^\e\|\\
     &\le C\e^{-2}\|W^\e\|^2\|\nabla W^\e\|^2\|\pp_xU^\e\|^2+\frac{1}{16}\zeta\|\nabla\pp_xU^\e\|^2+\frac{1}{16}\e\|\nabla\pp_xW^\e\|^2.
\end{align*}
Similar to the estimate of $I_5$ and $I_6$, we have
\begin{align*}
    J_5+J_6\le C\|\pp_xW^\e\|^2+\frac{1}{16}\zeta\|\nabla\pp_xU^\e\|^2+\|\pp_xG\|^2.
\end{align*}
Substituting the estimates of $J_1,\cdots,J_6$ into \eqref{eq_W_x_L2}, we have 
\begin{align}\label{eq_W_x_L2_1}
    \notag &\frac{1}{2}\frac{\dd}{\dd t}\left\|\pp_xW^\e \right\|^2+\frac{1}{2}\e\|\nabla\pp_xW^\e\|^2+4\zeta\|\pp_xW^\e\|\\ \notag
    \leq&\,  C\left(\left\|\left(\nabla w^I,\pp_xw^b,\langle z\rangle\pp_z w^b\right) \right\|^2_{L^\infty} + \e^{-2}\|W^\e\|^2\|\nabla W^\e\|^2 + 1\right) 
     \left\| (\pp_xW^\e,\pp_xU^\e )\right\|^2\\ 
     &+\frac{3}{16}\zeta\|\nabla\pp_xU^\e\|^2+\|\pp_xG\|^2 + C\e^{-1}\|(U^\e,W^\varepsilon)\|^2\|\pp_xw^a\|_{L^\infty}^2.
\end{align}
Summing \eqref{eq_U_x_L2_1} and \eqref{eq_W_x_L2_1} up, using Lemmas \ref{the estimate of approximate solution}-\ref{estimate of G} and \ref{lem_L_infty_L_2}, we have 
\begin{align*}
    \notag &\frac{\dd}{\dd t}\left\|\pp_x\left(U^\e,W^\e \right) \right\|^2+\zeta\|\nabla\pp_xU^\e\|^2+\e\|\nabla\pp_xW^\e\|^2 \\
    \notag&\le C\left( \varepsilon^{\frac{1}{2}}\left\|  \nabla W^\e \right\|^2+1\right)  \left\|\pp_x\left(U^\e,W^\e \right) \right\|^2+ C\e^{\frac{3}{2}},
\end{align*}
which together with Lemma \ref{lem_L_infty_L_2} and Gronwall inequality, implies 
 \begin{align*}
    \left\| \left(\pp_xU^\e,\pp_xW^\e\right)\right\|^2+\zeta \int_0^T\|\nabla \pp_xU^\e\|^2\dd t+\e\int_0^T\|\nabla \pp_xW^\e\|^2\dd t\le C_2e^{C_2T}\e^\frac{3}{2}.
\end{align*}
The proof is complete. 
\end{proof}
Next, we give the $L^\infty_TL^2_{xy}$ estimate of $(\partial_t U^\varepsilon,\partial_t W^\varepsilon )$.
\begin{lemma}\label{lem_t_L2}
Under the assumptions of Theorem \ref{thm1}, for any $0 < \e < 1$, there exists a constant $C$ independent of $\e$, such that
    \begin{align*}
   \left\| \left(\pp_tU^\e,\pp_tW^\e\right)\right\|^2+\zeta \int_0^T\|\nabla \pp_tU^\e\|^2\dd t+\e\int_0^T\|\nabla \pp_tW^\e\|^2\dd t\le C \e^\frac{3}{2}.
\end{align*}
\end{lemma}

\begin{proof} 
Multiplying $\pp_t\eqref{error equations}_1$ by $\pp_tU^\e$ and integrating the result by parts, we get
\begin{align}\label{eq_U_epsilon_t_L_2}
    \notag \frac{1}{2}\frac{\dd}{\dd t}\left\|\pp_tU^\e \right\|^2+\left(\e+\zeta \right)\left\|\nabla\pp_tU^\e\right\|^2=&~\left\langle u^a \otimes \pp_tU^\e,\nabla \pp_tU^\e\right\rangle+\left\langle U^\e\otimes \pp_tu^a,\nabla \pp_tU^\e\right\rangle\\ \notag
    &+\left\langle\pp_t u^a \otimes U^\e,\nabla \pp_tU^\e\right\rangle-\left\langle\pp_tU^\e\cdot \nabla U^\e, \pp_tU^\e\right\rangle\\ 
    &+2\zeta\left\langle \nabla^\perp\pp_tW^\e,\pp_tU^\e\right\rangle+\left\langle\pp_tF,\pp_tU^\e \right\rangle \notag \\
    &=\sum_{i=1}^{6}M_i.
\end{align}
For $M_1$, we have  
\begin{align*}
    M_1 \le \left\|u^a \right\|_{L^\infty}\left\|\pp_tU^\e \right\|\left\| \nabla\pp_tU^\e\right\|\le C \left\| u^a\right\|^2_{L^\infty}\left\| \pp_tU^\varepsilon\right\|^2+\frac{1}{16}\zeta\left\| \nabla\pp_tU^\varepsilon\right\|^2.
\end{align*}
For $M_2$, we have
\begin{align*}
   M_2=\notag \left\langle U^\e\otimes \pp_tu^a,\nabla \pp_tU^\e\right\rangle
    \le \left\|U^\e \right\|\left\|\pp_tu^a \right\|_{L^\infty}\left\| \nabla\pp_t U^\e\right\|
    \le C\left\|\pp_tu^a \right\|^2_{L^\infty}\left\|U^\varepsilon \right\|^2+\frac{1}{16}\zeta\left\| \nabla\pp_tU^\varepsilon\right\|^2.
\end{align*}
and
\begin{align*}
  M_3 \le C\left\| \pp_t u^a \right\|^2_{L^\infty}\left\|U^\varepsilon \right\|^2+\frac{1}{16}\zeta\left\| \nabla\pp_tU^\varepsilon\right\|^2.
\end{align*}
Using the Ladyzhenskaya inequality, we have
\begin{align*}
   \notag M_4= -\left\langle\pp_tU^\e\cdot \nabla U^\e, \pp_tU^\e\right\rangle
    &   \le\|\pp_tU^\e\|_{L^4}\|\pp_tU^\e \|_{L^4}\|\nabla U^\e\|\\
   \notag  &\le\|\pp_tU^\e\|\|\nabla\pp_tU^\e\| \|\nabla U^\e\|\\
    &\le C\|\nabla U^\e\|^2\|\pp_tU^\e\|^2+\frac{1}{16}\zeta\|\nabla\pp_tU^\e\|^2.
\end{align*}
For the last terms, we have
\begin{align*}
 \notag
   M_5+M_6&\le2\zeta\left\| \pp_tW^\e\right\|\left\| \nabla\pp_tU^\e\right\|+\left\| \pp_tF\right\|\left\| \pp_tU^\e\right\|\\
   &\le C\|\pp_tW^\e\|^2+C\|U^\e\|^2+\frac{1}{16}\zeta\|\nabla\pp_tU\|^2+\|\pp_tF\|^2.
\end{align*}
Substituting the estimates of $M_1,\cdots,M_6$ into \eqref{eq_U_epsilon_t_L_2}, we have 
\begin{align}\label{eq_U_epsilon_t_L_2_1}
    \notag \frac{1}{2}\frac{\dd}{\dd t}\left\|\pp_tU^\e \right\|^2+ \frac{11}{16}\zeta \left\|\nabla\pp_tU^\e\right\|^2 \leq &~C \left(\left\|u^a\right\|^2_{L^\infty}+\|\nabla U^\e\|^2 \right)\left\| (\pp_tU^\varepsilon,\partial_t W^\varepsilon)\right\|^2\notag \\
    &+C\left(\|\pp_tu^a \|^2_{L^\infty}+1\right)\left\|U^\varepsilon \right\|^2 + \|\pp_tF\|^2.
\end{align}
Next, we deal with the estimate of $\partial_tW^\varepsilon$. Taking $L^2$ inner product of $\pp_t\eqref{error equations}_2 $ with $\pp_tW^\e$, and using integration by parts,
we have 
\begin{align}\label{eq_W_epsilon_t_L_2}
    \notag \frac{1}{2}\frac{\dd}{\dd t}\left\|\pp_tW^\e \right\|^2+&\e\|\nabla\pp_tW^\e\|^2+4\zeta\|\pp_tW^\e\|\\ \notag
    =&~-\left\langle \pp_tU^\e\cdot \nabla w^a,\pp_tW^\e\right\rangle-\left\langle U^\e\cdot \nabla \pp_tw^a,\pp_tW^\e\right\rangle\\ \notag
    &~-\left\langle \pp_tu^a\cdot\nabla W^\e,\pp_tW^\e\right\rangle -\left\langle \pp_tU^\e\cdot\nabla W^\e,\pp_tW^\e\right\rangle\\
    &~+2\zeta\left\langle \nabla\times\pp_tU^\e,\pp_tW^\e\right\rangle+\left\langle\pp_tG,\pp_tW^\e \right\rangle \notag\\ 
    =&\sum_{i=1}^{6}N_i.
\end{align}
We handle $N_1$ as $M_1$ to get
\begin{align*}
    \notag N_1&\leq\|\pp_tU^\e\|\|(\pp_t w^b,\nabla w^I)\|_{L^\infty}\|\pp_tW^\e\|+\|\pp_y \pp_tU^\e\|\|\langle z\rangle \pp_z w^b\|_{L^\infty}\|\pp_tW^\e\|\\
    &\le \left\|\left(\nabla w^I,\pp_tw^b,\langle z\rangle\pp_z w^b\right) \right\|^2_{L^\infty}\left\| (\pp_tW^\e,\pp_tU^\e )\right\|^2+\frac{1}{16}\zeta\|\nabla\pp_tU^\e\|^2.
\end{align*}
For $N_2$ and $N_3$, we have 
\begin{align*}
     N_2=\left\langle\partial_x w^a   U^\e,\nabla\pp_tW^\e \right\rangle\le\|\pp_tw^a\|_{L^\infty}\|U^\e\|\|\nabla\pp_tW^\e\|\le C\e^{-1}\|U^\e\|^2\|\pp_tw^a\|_{L^\infty}^2+\frac{1}{16}\e\|\nabla\pp_tW^\e\|^2,
\end{align*}
and 
\begin{align*}
     N_3\le\|\pp_t u^a\|_{L^\infty}\|W^\e\|\|\nabla\pp_tW^\e\|\le C\e^{-1}\|W^\e\|^2\|\pp_t u^a\|_{L^\infty}^2+\frac{1}{16}\e\|\nabla\pp_tW^\e\|^2 .
\end{align*}
For $N_4$, using the Ladyzhenskaya inequality, we have 
\begin{align*}
     N_4&=\left\langle W^\e \pp_tU^\e ,\nabla\pp_tW^\e\right\rangle\le\|\pp_tU^\e\|_{L^4}\| W^\e\|_{L^4}\|\nabla\pp_tW^\e\|\\ \notag 
     &\le\|\pp_tU^\e\|^\frac{1}{2}\|\nabla\pp_tU^\e\|^\frac{1}{2}\|W^\e\|^\frac{1}{2}\|\nabla W^\e\|^\frac{1}{2}\|\nabla\pp_tW^\e\|\\
     &\le C\e^{-2}\|W^\e\|^2\|\nabla W^\e\|^2\|\pp_tU^\e\|^2+\frac{1}{16}\zeta\|\nabla\pp_tU^\e\|^2+\frac{1}{16}\e\|\nabla\pp_tW^\e\|^2.
\end{align*}
Similar to the estimate of $M_5$ and $M_6$, we have
\begin{align*}
    N_5+N_6\le C\|\pp_tW^\e\|^2+\frac{1}{16}\zeta\|\nabla\pp_tU^\e\|^2+\|\pp_tG\|^2.
\end{align*}
Substituting the estimates of $N_1,\cdots,N_6$ into \eqref{eq_W_epsilon_t_L_2}, we have 
\begin{align}\label{eq_W_epsilon_t_L_2_1}
    \notag &\frac{1}{2}\frac{\dd}{\dd t}\left\|\pp_tW^\e \right\|^2+\frac{1}{2}\e\|\nabla\pp_tW^\e\|^2+4\zeta\|\pp_tW^\e\|\\ \notag
    \leq&~C\left(\left\|\left(\nabla w^I,\pp_tw^b,\langle z\rangle\pp_z w^b\right) \right\|^2_{L^\infty} + \e^{-2}\|W^\e\|^2\|\nabla W^\e\|^2 + 1\right) 
     \left\| (\pp_tW^\e,\pp_tU^\e )\right\|^2\notag\\ 
    &+\frac{3}{16}\zeta\|\nabla\pp_tU^\e\|^2 +C\e^{-1}\|W^\e\|^2\|\pp_t u^a\|_{L^\infty}^2+ \|\pp_tG\|^2 .
\end{align}
Summing \eqref{eq_U_epsilon_t_L_2_1} and \eqref{eq_W_epsilon_t_L_2_1} up, using Lemmas \ref{the estimate of approximate solution}-\ref{estimate of G} and \ref{lem_L_infty_L_2}, we have 
\begin{align*}
    \notag &\frac{\dd}{\dd t}\left\|\pp_t\left(U^\e,W^\e \right) \right\|^2+\zeta\|\nabla\pp_tU^\e\|^2+\e\|\nabla\pp_tW^\e\|^2 \\
    \notag&\le C\left( \e^{\frac{1}{2}} \|\nabla W^\e\|^2 + 1\right)\left\| (\pp_tW^\e,\pp_tU^\e )\right\|^2 + C\varepsilon^{\frac{3}{2}},
\end{align*}
which together with Lemma \ref{lem_L_infty_L_2} and Gronwall inequality, implies 
\begin{align*}
    \left\| \left(\pp_tU^\e,\pp_tW^\e\right)\right\|^2+\zeta \int_0^T\|\nabla \pp_tU^\e\|^2\dd t+\e\int_0^T\|\nabla \pp_tW^\e\|^2\dd t\le C_3e^{C_3T}\e^\frac{3}{2}.
\end{align*} 
The proof is complete.
\end{proof}
Finally, we give the estimate of $(\partial_t\partial_x W^\varepsilon,\partial_t\partial_x U^\varepsilon)$.
\begin{lemma}
Under the assumptions of Theorem \ref{thm1}, for any $0 < \e < 1$, there exists a constant $C$ independent of $\e$, such that
    \begin{align*}
    \left\| \left(\pp_t\pp_xU^\e,\pp_t\pp_xW^\e\right)\right\|^2+\zeta \int_0^T\|\nabla \pp_t\pp_xU^\e\|^2\dd t+\e\int_0^T\|\nabla \pp_t\pp_xW^\e\|^2\dd t\le C_3e^{C_3T}\e^\frac{1}{2}.
\end{align*}
\end{lemma} 
\begin{proof}
Multiplying $\pp_t\pp_x\eqref{error equations}_1$ by $ \pp_t\pp_x U^\e$ and integrating the result by parts, we get
\begin{align}\label{eq_U_epsilon_tx_L_2}
    \notag &\frac{1}{2}\frac{\dd}{\dd t}\left\| \pp_t\pp_xU^\e \right\|^2+\left( \zeta+\e \right)  \left\| \nabla\pp_t\pp_xU^\e \right\|^2\\ \notag
   =&-\left\langle \pp_t\pp_xU^\e\cdot\nabla u^a,\pp_t\pp_xU^\e \right\rangle-\left\langle \pp_xU^\e\cdot\nabla \pp_tu^a,\pp_t\pp_xU^\e \right\rangle -\left\langle \pp_tU^\e\cdot\nabla \pp_xu^a,\pp_t\pp_xU^\e \right\rangle\\ \notag
    &-\left\langle U^\e\cdot\nabla \pp_t\pp_xu^a,\pp_t\pp_xU^\e \right\rangle-\left\langle \pp_t\pp_xu^a\cdot\nabla U^\e,\pp_t\pp_xU^\e \right\rangle-\left\langle \pp_xu^a\cdot\nabla \pp_t U^\e,\pp_t\pp_xU^\e \right\rangle\\ \notag
    & -\left\langle \pp_tu^a\cdot\nabla \pp_xU^\e,\pp_t\pp_xU^\e \right\rangle-\left\langle u^a\cdot\nabla \pp_t\pp_xU^\e,\pp_t\pp_xU^\e \right\rangle-\left\langle \pp_t\pp_xU^\e\cdot\nabla U^\e,\pp_t\pp_xU^\e \right\rangle\\ \notag
    &-\left\langle \pp_xU^\e\cdot\nabla \pp_tU^\e,\pp_t\pp_xU^\e \right\rangle -\left\langle \pp_tU^\e\cdot\nabla \pp_xU^\e,\pp_t\pp_xU^\e \right\rangle-\left\langle U^\e\cdot\nabla \pp_t\pp_xU^\e,\pp_t\pp_xU^\e \right\rangle\\ \notag
    &-\left\langle\nabla\pp_t\pp_xP^\varepsilon,\pp_t\pp_xU^\e  \right\rangle-2\zeta\left\langle \pp_t\pp_x W^\varepsilon,\nabla^\perp\cdot \pp_t\pp_xU^\e  \right\rangle+\left\langle \pp_t\pp_xF,\pp_t\pp_xU^\e \right\rangle\\
    =&\sum_{i=1}^{15}K_{i}.
\end{align}
For $K_1$, after integration by parts, we have
\begin{align*}
    K_1= & \,\left\langle  u^a\otimes \pp_t\pp_xU^\e ,\nabla \pp_t\pp_xU^\e \right\rangle
 \\
 \leq&\,\|\pp_t\pp_xU^\e\|\|u^a\|_{L^\infty}\|\nabla\pp_t\pp_xU^\e\|\le C\|\pp_t\pp_xU^\e\|^2\|u^a\|_{L^\infty}^2+\frac{1}{32}\zeta\|\nabla\pp_t\pp_xU^\e\|^2.
\end{align*}
Similarly, 
\begin{align*}
    K_2=\left\langle\pp_tu^a \otimes \pp_xU^\e ,\nabla\pp_t\pp_xU^\e \right\rangle\le C\left\| \pp_xU^\e \right\|^2\|\pp_tu^a\|^2_{L^\infty}+\frac{1}{32}\zeta\left\| \nabla\pp_t\pp_xU^\e \right\|^2,\\
    K_3=\left\langle  \pp_xu^a\otimes\pp_t U^\e,\nabla\pp_t\pp_xU^\e \right\rangle\le C\left\| \pp_t U^\e \right\|^2\|\pp_xu^a\|^2_{L^\infty}+\frac{1}{32}\zeta\left\| \nabla\pp_t\pp_xU^\e \right\|^2,\\
     K_4=\left\langle  \pp_t\pp_xu^a\otimes U^\e,\nabla\pp_t\pp_xU^\e \right\rangle\le C\left\| U^\e \right\|^2\|\pp_t \pp_xu^a\|^2_{L^\infty}+\frac{1}{32}\zeta\left\| \nabla\pp_t\pp_xU^\e \right\|^2.
\end{align*}
For $K_5$,  a direct estimate shows that 
\begin{align*}
    K_5\le \left\| \pp_t\pp_xu^a \right\|_{L^\infty}\left\| \nabla U^\e \right\|\left\|  \pp_t\pp_xU^\e\right\|\le C\left\| \nabla U^\e \right\|^2\|\pp_t\pp_xu^a\|^2_{L^\infty}+\left\| \pp_t\pp_xU^\e \right\|^2.
\end{align*}
Using integration by parts again, we have
\begin{align*}
    K_6&=\left\langle \pp_t U^\e\otimes\pp_xu^a,\nabla\pp_t\pp_xU^\e \right\rangle\le C\left\| \pp_t U^\e \right\|^2\|\pp_xu^a\|^2_{L^\infty}+\frac{1}{32}\zeta\left\| \nabla\pp_t\pp_xU^\e \right\|^2,\\
    K_7&=\left\langle \pp_xU^\e\otimes\pp_t u^a,\nabla\pp_t\pp_xU^\e \right\rangle\le C\left\| \pp_xU^\e \right\|^2\|\pp_t u^a\|^2_{L^\infty}+\frac{1}{32}\zeta\left\| \nabla\pp_t\pp_xU^\e \right\|^2,\\
    K_8&=0.
\end{align*}
Using the Ladyzhenskaya inequality, we have
\begin{align*}
    \notag K_9=&\left\langle  \pp_t\pp_xU^\e\otimes U^\e,\nabla\pp_t\pp_xU^\e \right\rangle\le\left\|  \pp_t\pp_xU^\e\right\|_{L^4} \left\| U^\e \right\|_{L^4} \left\|  \nabla\pp_t\pp_xU^\e\right\|\\ \notag 
    \le&\left\|  \pp_t\pp_xU^\e\right\|^\frac{1}{2} \left\|  \nabla\pp_t\pp_xU^\e\right\|^\frac{1}{2}  \left\| U^\e \right\|^\frac{1}{2} \left\| \nabla U^\e \right\|^\frac{1}{2}  \left\|  \nabla\pp_t\pp_xU^\e\right\|\\
    \le&C\left\| U^\e \right\|^2\left\| \nabla U^\e \right\|^2\left\|  \pp_t\pp_xU^\e\right\|^2 +\frac{1}{32}\zeta\|\nabla\pp_t\pp_xU^\e\|^2.
\end{align*}
Similarly, we have
\begin{align*}
    K_{10}\le&  \left\| \pp_xU^\e \right\|_{L^4}\left\| \pp_t U^\e \right\|_{L^4}\left\|\nabla\pp_t\pp_xU^\e  \right\|\\
    \le& C\left\| \pp_xU^\e \right\|^2\left\|  \nabla\pp_xU^\e \right\|^2+C\left\| \pp_t U^\e \right\|^2\left\|  \nabla\pp_tU^\e \right\|^2+\frac{1}{32}\zeta\left\| \nabla\pp_t\pp_xU^\e \right\|.
\end{align*}
and
\begin{align*}
    K_{11}\le C\left\| \pp_t U^\e \right\|^2\left\|  \nabla\pp_t U^\e \right\|^2+C\left\| \pp_xU^\e \right\|^2\left\|  \nabla\pp_xU^\e \right\|^2+\frac{1}{32}\zeta\left\| \nabla\pp_t\pp_xU^\e \right\|.
\end{align*}
Calculating directly, we have
\begin{align*}
    K_{12}=K_{13}=0
\end{align*}
Using the H\"older inequality, we have
\begin{align*}
    K_{14}+K_{15}\le C\left\| \pp_t\pp_xW^\e \right\|^2+C\left\| \pp_t\pp_xU^\e \right\|^2+\frac{1}{32}\zeta\left\| \nabla\pp_t\pp_xU^\e \right\|^2+\|\pp_t\pp_xF\|^2.
\end{align*}
Substituting the estimates of $K_1,\cdots,K_{15}$ into \eqref{eq_U_epsilon_tx_L_2}, we have 
\begin{align}\label{eq_U_epsilon_tx_L_2_1}
    \notag  &\frac{1}{2}\frac{\dd}{\dd t}\left\| \pp_t\pp_xU^\e \right\|^2+ \frac{5}{16} \zeta  \left\| \nabla\pp_t\pp_xU^\e \right\|^2 \\
   \leq &\,  C\left(\| u^a \|_{L^\infty}^2 + \left\| U^\e \right\|^2\left\| \nabla U^\e \right\|^2 +1 \right) \|(\pp_t\pp_xU^\e,\pp_t\pp_xW^\e)\|^2 + C\left\| \nabla U^\e \right\|^2\|\pp_t\pp_xu^a\|^2_{L^\infty} \notag \\
      &\, 
      +C\left(\left\| \pp_xU^\e \right\|^2\|\pp_tu^a\|^2_{L^\infty} + \left\| \pp_tU^\e \right\|^2\|\pp_xu^a\|^2_{L^\infty}+ \left\| U^\e \right\|^2\|\pp_t\pp_xu^a\|^2_{L^\infty}\right) \notag \\ 
      &\,+C\left(\| \pp_xU^\e \|^2 + \| \pp_t U^\e  \|^2 \right)\left(\|  \nabla\pp_xU^\e \|^2 +  \|  \nabla\pp_tU^\e  \|^2\right) +C \|\pp_t\pp_xF\|^2 \notag \\ 
      \leq &\,  C\left( \left\| \nabla U^\e \right\|^2 +1 \right) \|(\pp_t\pp_xU^\e,\pp_t\pp_xW^\e)\|^2    +C  \| (\nabla U^\e, \nabla\pp_xU^\e , \nabla\pp_tU^\e )  \|^2   +C \varepsilon^{\frac{3}{2}},
\end{align}
where Lemmas \ref{the estimate of approximate solution}-\ref{lem_t_L2} are used.

Next, multiplying $\pp_t\pp_x\eqref{error equations}_2$ by $ \pp_t\pp_x W^\e$ and integrating the result by parts, we get
\begin{align}\label{eq_W_epsilon_tx_L_2}
    \notag &\frac{1}{2}\frac{\dd}{\dd t}\left\| \pp_t\pp_xW^\e \right\|^2+ \e \left\| \nabla\pp_t\pp_xW^\e \right\|^2+4\zeta\|\pp_t\pp_xW^\e\|^2 \\ \notag
    =&-\left\langle \pp_t\pp_xU^\e\cdot\nabla w^a,\pp_t\pp_xW^\e \right\rangle-\left\langle \pp_xU^\e\cdot\nabla \pp_tw^a,\pp_t\pp_xW^\e \right\rangle -\left\langle \pp_tU^\e\cdot\nabla \pp_xw^a,\pp_t\pp_xW^\e \right\rangle\\ \notag
    &-\left\langle U^\e\cdot\nabla \pp_t\pp_xw^a,\pp_t\pp_xW^\e \right\rangle-\left\langle \pp_t\pp_xu^a\cdot\nabla W^\e,\pp_t\pp_xW^\e \right\rangle-\left\langle \pp_xu^a\cdot\nabla \pp_t W^\e,\pp_t\pp_xW^\e \right\rangle\\ \notag
    & -\left\langle \pp_tu^a\cdot\nabla \pp_xW^\e,\pp_t\pp_xW^\e \right\rangle-\left\langle u^a\cdot\nabla \pp_t\pp_xW^\e,\pp_t\pp_xW^\e \right\rangle-\left\langle \pp_t\pp_xU^\e\cdot\nabla W^\e,\pp_t\pp_xW^\e \right\rangle\\ \notag
    &-\left\langle \pp_xU^\e\cdot\nabla \pp_tW^\e,\pp_t\pp_xW^\e \right\rangle -\left\langle \pp_tU^\e\cdot\nabla \pp_xW^\e,\pp_t\pp_xW^\e \right\rangle-\left\langle U^\e\cdot\nabla \pp_t\pp_xW^\e,\pp_t\pp_xW^\e \right\rangle\\ \notag
    &+2\zeta\left\langle\nabla^\perp\cdot \pp_t\pp_xU^\e,\pp_t\pp_xW^\e  \right\rangle+\left\langle \pp_t\pp_xG,\pp_t\pp_xW^\e \right\rangle\\
    =&\sum_{i=1}^{14}L_{i}.
\end{align}
Using the Hardy inequality, we have
\begin{align*}
    L_1\le& \left\| \pp_t\pp_xU^\e \right\|\left\| (\nabla w^I,\pp_x w^B) \right\|_{L^\infty}\left\| \pp_t\pp_xW^\e \right\|+\left\| \nabla\pp_t\pp_xU^\e \right\|\left\| \langle z\rangle\pp_z w^B \right\|_{L^\infty}\left\| \pp_t\pp_xW^\e \right\|\\
    \le&C\left\| (\nabla w^I,\pp_x w^B,\langle z\rangle\pp_zw^B) \right\|_{L^\infty}^2\left\| \pp_t\pp_xW^\e \right\|^2+\left\|\pp_t\pp_xU^\e \right\|^2+\frac{1}{32}\zeta\left\| \nabla\pp_t\pp_xU^\e \right\|^2.
\end{align*}
Similarly, we can estimate $L_2$ and $L_3$ as follows
\begin{align*} 
    L_2\le& C\left\| (\nabla \pp_tw^I,\pp_x \pp_tw^B,\langle z\rangle\pp_z\pp_tw^B) \right\|_{L^\infty}^2\left\| \pp_t\pp_xW^\e \right\|^2+C\left\|\pp_xU^\e \right\|^2+ C\left\| \nabla\pp_xU^\e \right\|^2,\\
    L_3\le& C\left\| (\nabla \pp_xw^I,\pp_x \pp_xw^B,\langle z\rangle\pp_z\pp_xw^B) \right\|_{L^\infty}^2\left\| \pp_t\pp_xW^\e \right\|^2+C\left\|\pp_t U^\e \right\|^2+ C\left\| \nabla\pp_t U^\e \right\|^2. 
\end{align*}
For $L_4$, after integration by parts, we have
\begin{align*} 
    L_4\le& C\e^{-1}\|U^\e\|^2\|\pp_t\pp_xw^a\|^2_{L^\infty}+\frac{1}{32}\e\|\nabla\pp_t\pp_xW^\e\|^2.
\end{align*}
For $L_5$, a direct estimate shows that 
\begin{align*}
    L_5\le \left\| \pp_t\pp_xu^a \right\|_{L^\infty}\left\| \nabla W^\e \right\|\left\|  \pp_t\pp_xW^\e\right\|\le C\left\| \nabla W^\e \right\|^2\|\pp_t\pp_xu^a\|^2_{L^\infty}+\left\| \pp_t\pp_xW^\e \right\|^2.
\end{align*}
Similarly, after integration by parts, we have
\begin{align*}
    L_6&=\left\langle \pp_xu^a \pp_t W^\e,\nabla\pp_t\pp_xW^\e \right\rangle\le C\e^{-1}\left\| \pp_t W^\e \right\|^2\|\pp_xu^a\|^2_{L^\infty}+\frac{1}{32}\e\left\| \nabla\pp_t\pp_xW^\e \right\|^2.\\
    L_7&=\left\langle \pp_t u^a\pp_xW^\e,\nabla\pp_t\pp_xW^\e \right\rangle\le C\e^{-1}\left\| \pp_xW^\e \right\|^2\|\pp_t u^a\|^2_{L^\infty}+\frac{1}{32}\varepsilon\left\| \nabla\pp_t\pp_xW^\e \right\|^2.\\
    L_8&=0.
\end{align*}
For $L_9$, using the Ladyzhenskaya inequality, we have
\begin{align*}
    \notag L_9=&\left\langle  \pp_t\pp_xU^\e W^\e,\nabla\pp_t\pp_xW^\e \right\rangle\le\left\|  \pp_t\pp_xU^\e\right\|_{L^4} \left\| W^\e \right\|_{L^4} \left\|  \nabla\pp_t\pp_xW^\e\right\|\\ \notag 
    \le&\left\|  \pp_t\pp_xU^\e\right\|^\frac{1}{2} \left\|  \nabla\pp_t\pp_xU^\e\right\|^\frac{1}{2}  \left\| W^\e \right\|^\frac{1}{2} \left\| \nabla W^\e \right\|^\frac{1}{2}  \left\|  \nabla\pp_t\pp_xW^\e\right\|\\
    \le&C\e^{-2}\left\| W^\e \right\|^2\left\| \nabla W^\e \right\|^2\left\|  \pp_t\pp_xU^\e\right\|^2 +\frac{1}{32}\zeta\|\nabla\pp_t\pp_xU^\e\|^2+\frac{1}{32}\e\|\nabla\pp_t\pp_xW^\e\|^2.
\end{align*}
Similarly, we have
\begin{align*}
    L_{10}\le&  \left\| \pp_xU^\e \right\|_{L^4}\left\| \pp_t W^\e \right\|_{L^4}\left\|\nabla\pp_t\pp_xW^\e  \right\|\\
    \le& C\e^{-1}\left\| \pp_xU^\e \right\|^2\left\|  \nabla\pp_xU^\e \right\|^2+C\e^{-1}\left\| \pp_t W^\e \right\|^2\left\|  \nabla\pp_tW^\e \right\|^2+\frac{1}{32}\e\left\| \nabla\pp_t\pp_xW^\e \right\|,
\end{align*}
and
\begin{align*}
    L_{11}\le  C\e^{-1}\left\| \pp_t U^\e \right\|^2\left\|  \nabla\pp_t U^\e \right\|^2+C\e^{-1}\left\| \pp_xW^\e \right\|^2\left\|  \nabla\pp_xW^\e \right\|^2+\frac{1}{32}\e\left\| \nabla\pp_t\pp_xW^\e \right\|.
\end{align*}
Thanks to the divergence-free of $U^\varepsilon$, we have
\begin{align*}
    L_{12}=0.
\end{align*}
For the last two terms, using the H\"older inequality, we have
\begin{align*}
    L_{13}+L_{14}\le C\left\| \pp_t\pp_xW^\e \right\|^2+\frac{1}{32}\zeta\left\| \nabla\pp_t\pp_xU^\e \right\|^2+\|\pp_t\pp_xG\|^2.
\end{align*}
Substituting the estimates of $L_1,\cdots,L_{14}$ into \eqref{eq_W_epsilon_tx_L_2}, we have 
\begin{align}\label{eq_W_epsilon_tx_L_2_1}
    \notag &\frac{1}{2}\frac{\dd}{\dd t}\left\| \pp_t\pp_xW^\e \right\|^2+ \frac{1}{2} \e \left\| \nabla\pp_t\pp_xW^\e \right\|^2+4\zeta\|\pp_t\pp_xW^\e\|^2 \\ \notag
    \leq&\, C\Big(  \varepsilon^{\frac{1}{2}} \left\| \nabla W^\e \right\|^2 + 1\Big) \| (\pp_t\pp_x U^\e  ,\pp_t\pp_xW^\e  )\|^2  +  C (1+\varepsilon^{\frac{1}{2}}) \left\| (\nabla\pp_t U^\e,\nabla\pp_x U^\e) \right\|^2 \notag \\ 
     & + C\varepsilon^{\frac{1}{2}}\left\| (\nabla\pp_t W^\e,\nabla\pp_x W^\e) \right\|^2 + C\left\| \nabla W^\e \right\|^2 + \frac{3}{32}\zeta\left\| \nabla\pp_t\pp_xU^\e \right\|^2  + C\varepsilon^{\frac{1}{2}},
\end{align}
where Lemmas \ref{the estimate of approximate solution}-\ref{lem_t_L2} are used. 
Then, summing \eqref{eq_U_epsilon_tx_L_2_1} and \eqref{eq_W_epsilon_tx_L_2_1} up, using Lemmas \ref{the estimate of approximate solution}-\ref{lem_t_L2}, we have 
\begin{align*} 
     &\frac{\dd}{\dd t}\left\|\left(\pp_t\pp_xU^\e,\pp_t\pp_xW^\e \right) \right\|^2 +\zeta\|\nabla\pp_t\pp_xU^\e\|^2+\e\|\nabla\pp_t\pp_xW^\e\|^2 \\
     \leq&\, 
     C\Big( \left\| \nabla U^\e \right\|^2 + \varepsilon^{\frac{1}{2}} \left\| \nabla W^\e \right\|^2 + 1\Big) \| (\pp_t\pp_x U^\e  ,\pp_t\pp_xW^\e  )\|^2  +C  \| (\nabla U^\e, \nabla\pp_xU^\e , \nabla\pp_tU^\e )  \|^2 \\ 
     & + C\varepsilon^{\frac{1}{2}}\left\| (\nabla\pp_t W^\e,\nabla\pp_x W^\e) \right\|^2 + C\left\| \nabla W^\e \right\|^2    + C\varepsilon^{\frac{1}{2}},
\end{align*}
which together with Lemma \ref{lem_L_infty_L_2}-\ref{lem_t_L2} and Gronwall inequality, implies 
\begin{align*}
    \left\| \left(\pp_t\pp_xU^\e,\pp_t\pp_xW^\e\right)\right\|^2+\zeta \int_0^T\|\nabla \pp_t\pp_xU^\e\|^2\dd t+\e\int_0^T\|\nabla \pp_t\pp_xW^\e\|^2\dd t\le C_4e^{C_4T}\e^\frac{1}{2}.
\end{align*}
The proof is complete.
\end{proof}

\subsection{Convergence rate}\label{Convergence_rate}
\begin{lemma}\label{Linfty of U and W}
Under the assumptions of Theorem \ref{thm1}, for any $0 < \e < 1$, there exists a constant $C$ independent of $\e$, such that
    \begin{align*}
    \left\| U^\e\right\|_{L^\infty_{t}L^\infty_{xy}}\le C\e^\frac{7}{8} \text{~and~} \left\| W^\e\right\|_{L^\infty_{t}L^\infty_{xy}}\le C\e^\frac{5}{8}.
\end{align*}
\end{lemma}

\begin{proof}
Using Sobolev inequality, we have
\begin{align*}
     \left\| W^\e\right\|_{L^\infty_{t}L^\infty_{xy}}\le \left\| W^\e\right\|_{L^\infty_{t}L^\infty_{x}L^2_y}^\frac{1}{2} \left\| \pp_y W^\e\right\|_{L^\infty_{t}L^\infty_{x}L^2_y}^\frac{1}{2}  .
\end{align*}
where
\begin{align*}
      \left\| W^\e\right\|_{L^\infty_{tx}L^2_y}\le C\|W^\e\|_{L_t^\infty L^2_{xy}}^\frac{1}{2} \|W^\e\|_{L_t^\infty H^1_xL^2_{y}}^\frac{1}{2}\le C\e^{\left(\frac{5}{4}+\frac{3}{4}\right)\times\frac{1}{2}}=C\e  ,
\end{align*}
and
\begin{align*}
      \left\| \pp_yW^\e\right\|_{L^\infty_{tx}L^2_y}\le& C\|\pp_yW^\e\|_{L^2_{ty}L^\infty_x}^\frac{1}{2} \|\pp_t\pp_yW^\e\|_{L^2_{ty}L^\infty_x}^\frac{1}{2}\\
      \le&C\|\nabla W^\e\|_{L^2_{txy}}^\frac{1}{4}\|\nabla W^\e\|_{L^2_{ty}H^1_x}^\frac{1}{4}\|\nabla \pp_tW^\e\|_{L^2_{txy}}^\frac{1}{4}\|\nabla \pp_tW^\e\|_{L^2_{ty}H^1_x}^\frac{1}{4}\\
      \le& C\e^{\left(\frac{3}{4}+\frac{1}{4}+\frac{1}{4}-\frac{1}{4}\right)\times\frac{1}{4}}=C\e^\frac{1}{4}  .
\end{align*}
Hence, we get
\begin{align*}
    \left\| W^\e\right\|_{L^\infty_{t}L^\infty_{xy}}\le \left\| W^\e\right\|_{L^\infty_{t}L^\infty_{x}L^2_y}^\frac{1}{2} \left\| \pp_y W^\e\right\|_{L^\infty_{t}L^\infty_{x}L^2_y}^\frac{1}{2}\le C\e^{\left(1+\frac{1}{4}\right)\times\frac{1}{2}}=C\e^\frac{5}{8}.
\end{align*}
Similarly, we have
\begin{align*}
    \left\| U^\e\right\|_{L^\infty_{t}L^\infty_{xy}}\le C\e^\frac{7}{8}.
\end{align*}
The proof is complete.
\end{proof}

\subsection{Proof of Theorem \ref{thm1}}\label{sec_bl}
Using Lemmas \ref{regularity of wb0ub11}-\ref{regularity_of_wb2ub31}, \ref{the estimate of approximate solution},  $\ref{Linfty of U and W}$, Sobolev inequality, the definition of $U^\varepsilon$ and $W^\varepsilon$, we have
\begin{align} \label{eq_convergence_rate_u}
        &\left\| u(x,y,t)-u^{I,0}(x,y,t) \right\|_{L^\infty_TL^\infty_{xy}} \notag \\
         \lesssim &\, \varepsilon^{\frac{1}{2}}\| u^{I,1}  \|_{L^\infty_TL^\infty_{xy}} + \varepsilon \| u^{I,2}  \|_{L^\infty_TL^\infty_{xy}} + 
        \varepsilon^{\frac{1}{2}}\| u_1^{b,1}  \|_{L^\infty_TH^2_{xz}}
        +\varepsilon \| u^{b,2}  \|_{L^\infty_TH^2_{xz}}
        +\varepsilon^{\frac{3}{2}}\| u^{b,3}  \|_{L^\infty_TH^2_{xz}} \notag \\ 
        &+\varepsilon^{2}\| u_2^{b,4}  \|_{L^\infty_TH^2_{xz}}
        +\varepsilon^{\frac{3}{2}}\|S\|_{L^\infty_TH^2_{xy}}
        +\|U^\e(x,y,t)\|_{L^\infty_TL^\infty_{xy}} \le C\e^\frac{1}{2}, 
\end{align}
and       
\begin{align}\label{eq_convergence_rate_w}
        &\left\| w(x,y,t)-w^{I,0}(x,y,t)-w^{b,0}\left(x,\frac{y}{\sqrt{\e}},t\right) \right\|_{L^\infty_TL^\infty_{xy}} \notag \\
        \lesssim&\,\varepsilon^{\frac{1}{2}}\| w^{I,1}  \|_{L^\infty_TL^\infty_{xy}} + \varepsilon \| w^{I,2}  \|_{L^\infty_TL^\infty_{xy}} 
        +\varepsilon^{\frac{1}{2}}\| w^{b,1}  \|_{L^\infty_TH^2_{xz}}
        +\varepsilon \| w^{b,2}  \|_{L^\infty_TH^2_{xz}}
        +\|W^\e(x,y,t)\|_{L^\infty_TL^\infty_{xy}} \notag \\ 
        \le&\, C\e^\frac{1}{2}.
\end{align}
The proof of Theorem \ref{thm1} is complete.
\subsection{Proof of Theorem \ref{thm2}}
We divide the proof of Theorem \ref{thm2} into the following three Lemmas.
\begin{lemma}\label{lem_boundarylayer}
     Under the assumptions of Theorem $\ref{thm1}$, we have 
\begin{align*} 
 \liminf\limits_{\e\to 0}\left\|u-u^{I,0} \right\|_{L^\infty(0,T ;L^\infty(\mathbb{R}^2_+))} = 0. 
\end{align*} 
Moreover, 
\begin{align*}
     \liminf\limits_{\e\to 0}\left\|w-w^{I,0}  \right\|_{L^\infty(0,T ;L^\infty(\mathbb{R}^2_+))}>0,
\end{align*}
if and only if 
$$  
w^{I,0}(x,0,t)\neq 0,~~~\text{for some}~~t\in [0,T].
$$
\begin{proof}
    From \eqref{eq_convergence_rate_u}, we have 
    \begin{equation*}
        0\leq \| u(x,y,t)-u^{I,0}(x,y,t)  \|_{L^\infty_TL^\infty_{xy}} \leq C\varepsilon^{\frac{1}{2}} \to 0,~~\text{as}~\varepsilon\to 0,
    \end{equation*}
    i.e. 
    \begin{equation}\label{eq_Bl_eq_u}
        \lim\limits_{\e\to 0}\left\|u-u^{I,0} \right\|_{L^\infty(0,T ;L^\infty(\mathbb{R}^2_+))}= 0.
    \end{equation}
    Moreover, noticing that $w^{b,0}$ is non-trivial if and only if  $w^{I,0}(x,0,t)\neq 0$ for some $t\in (0,T]$  by \eqref{wb0}. Then, 
    using \eqref{eq_convergence_rate_w}, we have 
\begin{align*}
     &\| w(x,y,t)-w^{I,0}(x,y,t)  \|_{L^\infty_TL^\infty_{xy}}   \\ 
     =&  \,\left\| w(x,y,t)-w^{I,0}(x,y,t)-w^{b,0}\left(x,\frac{y}{\sqrt{\e}},t\right)  + w^{b,0}\left(x,\frac{y}{\sqrt{\e}},t\right) \right\|_{L^\infty_TL^\infty_{xy}} \\ 
     \geq & \,\,\|w^{b,0}\|_{L^\infty_TL^\infty_{xy}}  - \left\| w(x,y,t)-w^{I,0}(x,y,t)-w^{b,0}\left(x,\frac{y}{\sqrt{\e}},t\right)  \right\|_{L^\infty_TL^\infty_{xy}}~~\to \|w^{b,0}\|_{L^\infty_TL^\infty_{xy}},~~~\text{as}~\varepsilon\to 0,
\end{align*}
which implies 
    \begin{equation}\label{eq_Bl_eq_w}
        \lim\limits_{\e\to 0}\left\|w-w^{I,0} \right\|_{L^\infty(0,T ;L^\infty(\mathbb{R}^2_+))} > 0,
    \end{equation}
    if and only if 
$$  
w^{I,0}(x,0,t)\neq 0,~~~\text{for some}~~t\in [0,T].
$$
The proof is complete.
\end{proof}
\end{lemma}
To prove \eqref{lim} with $\delta(\varepsilon)$ satisfying $\lim\limits_{\e\to 0}\delta^{-1}\e^\frac{1}{2}=0$, we have the following Lemma for the boundary layer correctors $(u^b,w^b)$.
\begin{lemma}\label{BL-b}
 Under the assumptions of Theorem $\ref{thm1}$, for any  non-negative smooth function $\delta(\varepsilon)$ with $ \lim\limits_{\e\to 0}\delta^{-1}\e^\frac{1}{2}=0$, we have
    \begin{align*}
        \lim\limits_{\e\to 0}\left\| \left( u^b,w^b\right) \right\|_{L^\infty(0,T; L^\infty(\mathbb{R}\times(\delta,+\infty)))}=0,~~\lim\limits_{\e\to 0}\left\| \left( u^I-u^{I,0},w^I-w^{I,0}\right) \right\|_{L^\infty(0,T; L^\infty(\mathbb{R}\times(\delta,+\infty)))}=0.
    \end{align*}
\end{lemma}

\begin{proof}
To begin with, by the definition of $u^b$, one can easily find that 
\begin{align*}
         \lim\limits_{\e\to 0}\left\|  u^b  \right\|_{L^\infty(0,T; L^\infty(\mathbb{R}\times(\delta,+\infty)))}=0.
    \end{align*}
Next, to handle $w^b$,
without loss of generality, we can assume that $0<\varepsilon<1$. Then, from the assumption, there exists a constant $C>0$ such that $\delta > C\varepsilon^{\frac{1}{2}}$. Thus,
for any $(x,y,t)\in \mathbb{R}\times(\delta,+\infty)\times(0,T)$, we have 
    \begin{align*}
        \left\|w^{b,0}\left(x,\frac{y}{\sqrt{\e}},t\right) \right\|_{L^\infty(0,T; L^\infty_{xy}(\mathbb{R}\times(\delta,+\infty)))} 
         \leq&\, \frac{ \varepsilon^{\frac{1}{2}}}{\delta}\left\|\frac{y}{\sqrt{\e}}w^{b,0}\left(x,\frac{y}{\sqrt{\e}},t\right) \right\|_{L^\infty(0,T; L^\infty_{xy}(\mathbb{R}\times(\delta,+\infty)))}      \\ 
        \leq&\,\frac{ \varepsilon^{\frac{1}{2}}}{\delta}\left\|\langle z\rangle w^{b,0}\left(x,z,t\right) \right\|_{L^\infty(0,T; L^\infty_{xz}(\mathbb{R}\times(C\varepsilon^{\frac{1}{2}},+\infty)))}  \\ 
        \leq&\,\frac{ \varepsilon^{\frac{1}{2}}}{\delta}\left\|\langle z\rangle w^{b,0} \right\|_{L^\infty_TH^1_xH^1_z}~\to~0,~~\text{as} ~~\varepsilon\to 0,
    \end{align*}
    which, together with the definition of $w^b$, implies   
    \begin{align*}
         \lim\limits_{\e\to 0}\left\|  w^b  \right\|_{L^\infty(0,T; L^\infty(\mathbb{R}\times(\delta,+\infty)))}=0.
    \end{align*} 
    Finally, noticing that 
    \begin{equation*}
         u^I-u^{I,0} = \varepsilon^{\frac{1}{2}}(u^{I,1} + \varepsilon^{\frac{1}{2}}u^{I,2}),~~~ w^I-w^{I,0} = \varepsilon^{\frac{1}{2}}(w^{I,1} + \varepsilon^{\frac{1}{2}}w^{I,2}),
    \end{equation*}
    One can easily prove that 
    \begin{equation*}
        \lim\limits_{\e\to 0}\left\| \left( u^I-u^{I,0},w^I-w^{I,0}\right) \right\|_{L^\infty(0,T; L^\infty(\mathbb{R}\times(\delta,+\infty)))}=0.
    \end{equation*}
    The proof is complete.
\end{proof}
\begin{lemma}\label{BL-I}
Under the assumptions of Theorem $\ref{thm1}$, for any  non-negative smooth function $\delta(\varepsilon)$ with $ \lim\limits_{\e\to 0}\delta^{-1}\e^\frac{1}{2}=0$, we have
    \begin{align*} 
        \lim\limits_{\e\to 0}\left\| \left( u-u^{I,0},w-w^{I,0}\right) \right\|_{L^\infty(0,T; L^\infty(\mathbb{R}\times(\delta,+\infty)))}=0.
\end{align*}
\end{lemma}

\begin{proof}
Combining Lemmas  \ref{the estimate of approximate solution}, \ref{Linfty of U and W}, \ref{BL-b} and the definition of $(U^\varepsilon,W^\varepsilon)$, we have
\begin{align*}
    &\lim\limits_{\e\to 0}\left\| \left( u-u^{I,0},w-w^{I,0}\right) \right\|_{L^\infty(0,T; L^\infty(\mathbb{R}\times(\delta,+\infty)))}\\
    =&\,\lim\limits_{\e\to 0}\left\| \left( u^I-u^{I,0}+u^b+\e^\frac{3}{2}S+U^\e,w^I-w^{I,0}+w^b+W^\e\right) \right\|_{L^\infty(0,T; L^\infty(\mathbb{R}\times(\delta,+\infty)))}\\
    \le&\,\lim\limits_{\e\to 0}\left\| \left( u^I-u^{I,0}, w^I-w^{I,0}\right)\right\|_{L^\infty(0,T; L^\infty(\mathbb{R}\times(\delta,+\infty)))}+\lim\limits_{\e\to 0}\left\| \left(u^b,w^b\right)\right\|_{L^\infty(0,T; L^\infty(\mathbb{R}\times(\delta,+\infty)))}\\
    &+\lim\limits_{\e\to 0}\left\| \left(U^\varepsilon,W^\varepsilon\right)\right\|_{L^\infty(0,T; L^\infty(\mathbb{R}\times(\delta,+\infty)))} + \lim\limits_{\e\to 0}\varepsilon^{\frac{3}{2}}\left\| S \right\|_{L^\infty(0,T; L^\infty(\mathbb{R}\times(\delta,+\infty)))}  \\
    =&\,  0.
\end{align*}
The proof is complete.
\end{proof}

\section*{Appendix}

    \appendix

\section{Derivation of inner and outer profiles}\label{Appendix A}
In this section, we will give a formal derivation  of the inner and outer profiles with the corresponding initial and boundary conditions (see Chapter 4 of \cite{Holmes_2013} or Appendix A of \cite{Wang and Wen} for more detailed illustrations). \\[2mm]
 {\bfseries Step 1. The initial and boundary conditions.}
Substituting \eqref{expansions} into  initial and boundary conditions \eqref{uepsilon_ini_bou_cond}, we find the initial and boundary conditions should satisfy
\begin{align}\label{initial}
&(u^{I,0},w^{I,0})|_{t=0}=(u_0,w_0),\,\,
(u^{I,j},w^{I,j})|_{t=0}=0,\,j\ge1,~~~~
(u^{b,i},w^{b,i})|_{t=0}=0,\,i\ge 0,  
\end{align}
and
\begin{equation}\label{con1}  
     u^{I,i}(x,0,t)+u^{b,i}(x,0,t)=0,~~~~~w^{I,i}(x,0,t)+w^{b,i}(x,0,t)=0,~~ \forall i \ge 0.  
\end{equation} 
\noindent {\bfseries Step 2. Equations of  leading order profiles.}
Plugging $\eqref{expansions}$ into $\eqref{uepsilon}$, we have
\begin{align} 
  \displaystyle
&\pp_t  \sum_{j=0}^{+\infty}\e^\frac{j}{2}\left(u^{I,j}+u^{b,j}\right)+\left(\sum_{j=0}^{+\infty}\e^\frac{j}{2}\left(u^{I,j}+u^{b,j}\right)\cdot\nabla \right)\sum_{k=0}^{+\infty}\e^\frac{k}{2}\left(u^{I,k}+u^{b,k}\right)\notag\\  \displaystyle
&~~~~+\nabla \sum_{j=0}^{+\infty}\e^\frac{j}{2}\left(p^{I,j}+p^{b,j}\right)-\left(\e+\zeta \right)\D \sum_{j=0}^{+\infty}\e^\frac{j}{2}\left(u^{I,j}+u^{b,j}\right)=-2\zeta\nabla^\perp \sum_{j=0}^{+\infty}\e^\frac{j}{2}\left(w^{I,j}+w^{b,j}\right), \label{eq_u_I_b_s}\\[2mm] 
&\pp_t  \sum_{j=0}^{+\infty}\e^\frac{j}{2}\left(w^{I,j}+w^{b,j}\right) +\left(\sum_{j=0}^{+\infty}\e^\frac{j}{2}\left(u^{I,j}+u^{b,j}\right)\cdot\nabla \right)\sum_{k=0}^{+\infty}\e^\frac{k}{2}\left(w^{I,k}+w^{b,k}\right)\notag\\  
&~~~~+4\zeta \sum_{j=0}^{+\infty}\e^\frac{j}{2}\left(w^{I,j}+w^{b,j}\right)-\e\D \sum_{j=0}^{+\infty}\e^\frac{j}{2}\left(w^{I,j}+w^{b,j}\right)=2\zeta\nabla^\perp\cdot \sum_{j=0}^{+\infty}\e^\frac{j}{2}\left(u^{I,j}+u^{b,j}\right),  \label{eq_w_I_b_s} \\ \text{and} \notag \\[2mm] \displaystyle
&\di \sum_{j=0}^{+\infty}\e^\frac{j}{2}\left(u^{I,j}+u^{b,j}\right)=0. 
\label{eq_divu_I_b_s} 
\end{align} 
Formally, let $z\to +\infty$, we get 
\begin{equation}\label{eq_IIII}
    \begin{cases}
          \displaystyle  \partial_t u^{I,j} + \sum_{\ell=0}^j u^{I,\ell}\cdot\nabla u^{I,j-\ell} + \nabla p^{I,j} - \zeta\Delta u^{I,j}-2\zeta\nabla^\perp w^{I,j} = \Delta u^{I,j-2}, \\[1mm] 
          \displaystyle  \partial_t w^{I,j} + \sum_{\ell=0}^j u^{I,\ell}\cdot\nabla w^{I,j-\ell} + 4\zeta w^{I,j} + 2\zeta\nabla^\perp u^{I,j} = \Delta w^{I,j-2}, \\[1mm] 
           \displaystyle \mathrm{div} u^{I,j} = 0,
    \end{cases}
\end{equation}
for $j\geq 0$, where $u^{I,-1} = u^{I,-2}  = 0$ and $w^{I,-1} = w^{I,-2}  = 0$. 

\begin{lemma}\label{lem_b0}
    The zeroth order outer profiles $(u^{I,0},p^{I,0},w^{I,0})$  satisfies the limit problem $\eqref{I0}-\eqref{I0_cond}$.
    The zeroth order inner profiles $u^{b,0}$ and $p^{b,0}$ vanish identically, i.e., 
    \begin{equation}\label{eq_u_p_b0}
        u^{b,0} = 0,~p^{b,0} = 0.
    \end{equation}
    The zeroth order inner profile  $w^{b,0}$ satisfies problem \eqref{wb0}.
\end{lemma}
\begin{proof}
    Near the boundary, subtracting \eqref{eq_IIII}$_1$ from \eqref{eq_u_I_b_s}, then using the fact $y = \sqrt{\varepsilon}z$ and Taylor's formula
    $$
    f(x,y,t) = f(x,\sqrt{\varepsilon}z,t) = \sum_{\ell=0}^{\infty}\frac{1}{\ell!}(\sqrt{\varepsilon}z)^\ell\partial_y^\ell f(x,0,t),
    $$
    for the outer layer profiles $(u^{I,j},p^{I,j},w^{I,j})$, we get
 \begin{equation*}
        \sum_{j=-2}^\infty \varepsilon^{j/2} \mathcal{F}^j(x,z,t) = 0,
    \end{equation*}
    where
    \begin{align*} 
            \mathcal{F}^{-2} =&\,  -\zeta \partial_z^2 u^{b,0},\\[2mm]
            \mathcal{F}^{-1}=&\,   (u_2^{I,0}(x,0,t) + u_2^{b,0})\partial_z u^{b,0} + (0,\partial_zp^{b,0})^\top -\zeta \partial_z^2u^{b,1} - (2\zeta\partial_z w^{b,0},0)^\top,\\[2mm]
            \mathcal{F}^{0}=&\,   \partial_t u^{b,0} + u^{b,0}\cdot \nabla u^{I,0}(x,0,t) + (u_1^{I,0}(x,0,t) + u^{b,0}_1)\partial_xu^{b,0} +  (u_2^{I,0}(x,0,t) + u^{b,0}_2)\partial_zu^{b,1}\\ 
            &+(u_2^{I,1}(x,0,t) + u^{b,1}_2)\partial_zu^{b,0} + (\partial_xp^{b,0},\partial_zp^{b,1})^\top - \zeta\partial_x^2 u^{b,0} -\zeta\partial_z^2u^{b,2} - \partial_z^2u^{b,0}\\
            &-(2\zeta\partial_zw^{b,1} ,-2\zeta\partial_xw^{b,0} )^\top + z\partial_yu^{I,0}_2(x,0,t)\partial_zu^{b,0},\\[2mm]
            \cdots\cdots
    \end{align*}
    Moreover, from \eqref{eq_divu_I_b_s} and \eqref{eq_IIII}$_2$, we have 
    \begin{equation}\label{eq_divfree_b}
        \partial_x u_1^{b,j} + \partial_z u_2^{b,j+1} = 0,~~~\forall j\geq 0. 
    \end{equation} 
    Hence, $\mathcal{F}^{-2} = 0$ and \eqref{eq_divfree_b} yield
    \begin{eqnarray}\label{eq_u_B_00}
     u^{b,0} = 0,~~~u_2^{b,1} = 0,
    \end{eqnarray}
    which together with the boundary condition \eqref{con1} implies 
    \begin{equation}\label{eq_boundary_u_I_0}
        u^{I,0}(x,0,t) = 0,~~~u_2^{I,1}(x,0,t) = 0.
    \end{equation}
Letting $j = 0$ in \eqref{eq_IIII}, combining with \eqref{initial} and \eqref{eq_boundary_u_I_0}, we find that $(u^{I,0},p^{I,0},w^{I,0})$  satisfies the limit problem $\eqref{I0}-\eqref{I0_cond}$. Next, from $\mathcal{F}^{-1} = 0$ and \eqref{eq_u_B_00}, we have 
\begin{equation*}
    \begin{pmatrix}
        0\\
        \partial_z p^{b,0}
    \end{pmatrix}
    -\begin{pmatrix}
        \zeta\partial_z^2u^{b,1}_1\\
        0
    \end{pmatrix}
    -\begin{pmatrix}
        2\zeta\partial_z w^{b,0}\\
        0
    \end{pmatrix}
    =0,
\end{equation*}
which implies 
\begin{equation}\label{eq_p_b_0_u_w}
    p^{b,0} = 0,~~\partial_z u^{b,1}_1 + 2w^{b,0} = 0.
\end{equation}
Now, we are in the position to deduce the equation of $w^{b,0}$. Similar to the derivation of $\mathcal{F}^{j}$, from \eqref{eq_u_I_b_s}  and \eqref{eq_IIII}$_2$, we obtain 
\begin{equation*}
        \sum_{j=-1}^\infty \varepsilon^{j/2} \mathcal{G}^j(x,z,t) = 0,
    \end{equation*}
    where
    \begin{align*} 
            \mathcal{G}^{-1} =\,&  (u_2^{I,0}(x,0,t) + u^{b,0}_2)\partial_zw^{b,0} +2\zeta\partial_z u_1^{b,0},\\[2mm]
            \mathcal{G}^{0}=\,&  \partial_tw^{b,0} + u^{b,0}\cdot\nabla w^{I,0}(x,0,t) + (u_1^{I,0}(x,0,t) + u^{b,0}_1)\partial_xw^{b,0} +(u_2^{I,0}(x,0,t) + u^{b,0}_2)\partial_zw^{b,1}\\[2mm]
            &+(u_2^{I,1}(x,0,t) + u^{b,1}_2)\partial_zw^{b,0} +4\zeta w^{b,0} -\partial_z^2w^{b,0} -2\zeta\partial_x u_2^{b,0} +2\zeta\partial_z u_1^{b,1} + z\partial_yu^{I,0}_2(x,0,t)\partial_zw^{b,0},\\
            \cdots\cdots 
    \end{align*}
From $\mathcal{G}^{0} = 0,$ \eqref{eq_u_B_00}--\eqref{eq_p_b_0_u_w}, the initial condition \eqref{initial} and the boundary condition \eqref{con1} , we find that 
$w^{b,0}$ satisfies the following problem 
\begin{equation*} 
    \begin{cases}
        \pp_tw^{b,0}-\pp_z^2 w^{b,0}=0,\\[1mm]
        w^{b,0}(x,z,0)=0,~w^{b,0}(x,0,t)=-w^{I,0}(x,0,t).
    \end{cases}
\end{equation*}
The proof is complete.
\end{proof} 
\noindent {\bfseries Step 3. Equations of  first order profiles.}
From \eqref{eq_divfree_b}, \eqref{eq_u_B_00},  \eqref{eq_p_b_0_u_w}  and , we have the following Corollary.
\begin{corollary}\label{coro_u_b_1}
    The first order boundary layer profile $u^{b,1}$ satisfies 
    \begin{equation}\label{eq_u_b_1_app}
    u^{b,1}_1 = 2\int_z^{+\infty} w^{b,0}(x,s,t)\mathrm{d}s,~~~u^{b,1}_2 = 0. 
    \end{equation}
    As a consequence (using \eqref{con1}), we have the following boundary conditions:
    \begin{equation}\label{eq_u_I_1_boundary}
        u_{1}^{I,1} (x,0,t)= -2\int_0^{+\infty} w^{b,0}(x,s,t)\mathrm{d}s,~~~u_{2}^{I,1} (x,0,t)= 0. 
    \end{equation}
\end{corollary}
\begin{lemma}\label{lem_b1}
    The first order outer profiles $(u^{I,1},p^{I,1},w^{I,1})$  satisfies  problem $\eqref{I1}$.
    The first order inner profile $p^{b,1} = 0$. The first order inner profile $w^{b,1}$  satisfies problem \eqref{wb1}.
\end{lemma}
\begin{proof}
    Taking $j = 1$ in \eqref{eq_IIII}, we have $(u^{I,1},u^{I,1},u^{I,1})$  satisfies 
\begin{equation*}
    \begin{cases}
        \pp_t  u^{I,1}+\left( u^{I,1}\cdot\nabla\right)u^{I,0}+\left( u^{I,0}\cdot\nabla\right)u^{I,1}+\nabla p^{I,1}-\zeta\D u^{I,1}-2\zeta\nabla^\perp w^{I,1}=0,\\ 
        \pp_t  w^{I,1}+\left( u^{I,1}\cdot\nabla\right)w^{I,0}+\left( u^{I,0}\cdot\nabla\right)w^{I,1}+4\zeta w^{I,1}+2\zeta\nabla^\perp\cdot u^{I,1}=0,\\
        \di u^{I,1}=0,
    \end{cases}
\end{equation*}
which together with \eqref{initial}, \eqref{con1} and \eqref{eq_u_I_1_boundary} implies that $(u^{I,1},u^{I,1},u^{I,1})$  satisfies the linear problem \eqref{I1}.  Next, from $\mathcal{F}^{0} = 0$, Lemmas \ref{lem_b0}, \ref{lem_b1} and Corollary \ref{coro_u_b_1}, we have 
\begin{equation*}
    \begin{pmatrix}
        0\\
        \partial_z p^{b,1}
    \end{pmatrix}
    -\begin{pmatrix}
        \zeta\partial_z^2u^{b,2}_1\\
        \zeta\partial_z^2u^{b,2}_2
    \end{pmatrix}
    -\begin{pmatrix}
        2\zeta\partial_z w^{b,1}\\
        -2\zeta\partial_x w^{b,0}
    \end{pmatrix}
    =0,
\end{equation*}
which implies 
\begin{equation}\label{eq_u_w_2}
    \partial_z u^{b,2}_1 + 2w^{b,1} = 0.
\end{equation}
and 
\begin{equation}\label{eq_p_b_1}
      p^{b,1} = -\zeta\int_z^\infty\big(\partial_z^2u^{b,2}_2 -2 \partial_x w^{b,0}  \big)\mathrm{d}z =  \zeta\int_z^\infty\partial_x\big(\partial_zu^{b,1}_1 +2  w^{b,0}  \big)\mathrm{d}z = 0, 
\end{equation}
where \eqref{eq_divfree_b} and \eqref{eq_p_b_0_u_w} are used. Moreover, from $\mathcal{G}^{1} = 0$, Lemmas \ref{lem_b0}, \ref{lem_b1}, Corollary \ref{coro_u_b_1},  and \eqref{eq_u_w_2}, we have 
\begin{equation}\label{eq_w_b_1_orign}
    \begin{split}
        &\pp_t  w^{b,1}+\overline{u^{I,1}_1}\pp_x w^{b,0}+u^{b,1}_1\overline{\pp_x w^{I,0}}+u^{b,1}_1\pp_xw^{b,0}+\left(\overline{ u^{I,2}_2}+u^{b,2}_2 \right)\pp_zw^{b,0} \\ 
    &-\pp_z^2w^{b,1} +\overline{\pp_y^2u^{I,0}_2}\frac{1}{2}z^2\pp_zw^{b,0} +\overline{\pp_yu^{I,0}_1}z\pp_xw^{b,0}+\overline{\pp_yu^{I,1}_2}z\pp_zw^{b,0}
=0,
    \end{split}
\end{equation}
which together  \eqref{initial} and \eqref{con1} implies that $w^{b,1}$  satisfies problem \eqref{wb1}. The proof is complete.
\end{proof}
 
\noindent {\bfseries Step 4. Equations of  second order profiles.}
From \eqref{eq_divfree_b}, \eqref{eq_u_B_00}, \eqref{eq_boundary_u_I_0}, \eqref{eq_u_b_1_app} and \eqref{eq_u_w_2}, we have the following Corollary.
\begin{corollary}\label{coro_u_b_2}
    The first order boundary layer profile $u^{b,2}$ satisfies 
    \begin{equation}\label{eq_u_b_2}
    u^{b,2}_1 = 2\int_z^{+\infty} w^{b,1}(x,s,t)\mathrm{d}s,~~~u^{b,2}_2 = \int_z^{+\infty} \partial_xu_1^{b,1}(x,s,t)\mathrm{d}s = 2\int_z^\infty \int_\tau^{+\infty} \partial_xw^{b,0}(x,s,t)\mathrm{d}s\mathrm{d}\tau. 
    \end{equation}
    As a consequence (using \eqref{con1}), we have the following boundary conditions:
    \begin{equation}\label{eq_u_I_2_boundary}
        u_{1}^{I,2} (x,0,t)= -2\int_0^{+\infty} w^{b,1}(x,s,t)\mathrm{d}s,~~~u_{2}^{I,2} (x,0,t)= -2\int_0^\infty \int_\tau^{+\infty} \partial_xw^{b,0}(x,s,t)\mathrm{d}s\mathrm{d}\tau. 
    \end{equation}
\end{corollary}
\begin{lemma}\label{lem_b2}
    The second order outer profiles $(u^{I,2},p^{I,2},w^{I,2})$  satisfies  problem $\eqref{I2}$.
    The second order inner profile $p^{b,2} = 0$. The first order inner profile $w^{b,2}$  satisfies problem \eqref{wb2}. 
\end{lemma}
\begin{proof}
Taking $j = 2$ in \eqref{eq_IIII}, we find $(u^{I,2},p^{I,2},w^{I,2})$  satisfies 
\begin{equation*}
    \begin{cases}
        \pp_t  u^{I,2}+\left( u^{I,2}\cdot\nabla\right)u^{I,0}+\left( u^{I,1}\cdot\nabla\right)u^{I,1}+\left( u^{I,0}\cdot\nabla\right)u^{I,2}+\nabla p^{I,2} -\zeta\D u^{I,2}-2\zeta\nabla^\perp w^{I,2}=\D u^{I,0},\\ 
        \pp_t  w^{I,2}+\left( u^{I,2}\cdot\nabla\right)w^{I,0}+\left( u^{I,1}\cdot\nabla\right)w^{I,1}+\left( u^{I,0}\cdot\nabla\right)w^{I,2}+4\zeta w^{I,2}+2\zeta\nabla^\perp\cdot u^{I,2}= \D w^{I,0},\\
        \di u^{I,2}=0,
    \end{cases}
\end{equation*}
which together with \eqref{initial}, \eqref{con1} and \eqref{eq_u_I_2_boundary} implies that $(u^{I,2},p^{I,2},w^{I,2})$  satisfies the linear problem \eqref{I2}. 
Next, from $\mathcal{F}^{1} = 0$, Lemmas \ref{lem_b0}, \ref{lem_b1}, \ref{lem_b2}, Corollary \ref{coro_u_b_1} and \ref{coro_u_b_2}, we have 
\begin{equation}\label{eq_u_b_1_b_3}
\begin{pmatrix}
   \partial_t u^{b,1}_1\\
   0
\end{pmatrix}   + 
\begin{pmatrix}
   0\\
   \pp_zp^{b,2}
\end{pmatrix}  
    -\zeta\begin{pmatrix}
        \partial_x^2u_1^{b,1}+\partial_z^2u^{b,3}_1\\
         \partial_z^2u^{b,3}_2
    \end{pmatrix}
    - \begin{pmatrix}
        \partial_z^2u^{b,1}_1\\
         0
    \end{pmatrix}
    -\begin{pmatrix}
        2\zeta\partial_z w^{b,2}\\
        -2\zeta\partial_x w^{b,1}
    \end{pmatrix}
    =0.
\end{equation}
Then, from \eqref{eq_u_b_1_b_3}$_1$, we have  
\begin{equation}\label{eq_u_w_3}
    \partial_z u^{b,3}_1 + 2w^{b,2} = \frac{1}{\zeta}\int_z^\infty\left( \zeta \partial_x^2u_1^{b,1} +  \partial_z^2u_1^{b,1} -  \partial_t u^{b,1}_1\right)\mathrm{d}z = \frac{1}{\zeta}\int_z^\infty\left( \zeta \partial_x^2u_1^{b,1}  -  \partial_t u^{b,1}_1\right)\mathrm{d}z -\frac{1}{\zeta} \partial_z  u_1^{b,1}.
\end{equation}
Moreover, using the fact 
\begin{equation*}
    \partial_x u_1^{b,2} +  \partial_z u_2^{b,3}  = 0,
\end{equation*}
 and \eqref{eq_u_w_2}, we have from \eqref{eq_u_b_1_b_3}$_2$ that 
\begin{equation}\label{eq_p_b_2}
    p^{b,2} = - \zeta\int_z^{\infty} \big(\partial_z^2u^{b,3}_2 - 2\partial_x w^{b,1} \big) \mathrm{d}z = \zeta\int_z^{\infty} \partial_x\big(\partial_z u^{b,2}_1 + 2 w^{b,1} \big) \mathrm{d}z = 0. 
\end{equation}
Moreover, from $\mathcal{G}^{2} = 0$, Lemmas \ref{lem_b0}, \ref{lem_b1}, \ref{lem_b2}, Corollary \ref{coro_u_b_1}, \ref{coro_u_b_2}  and \eqref{eq_u_w_2}, we have 
\begin{equation}\label{eq_w_b_2_orign}
    \begin{split}
        &\pp_t  w^{b,2}+\overline{u^{I,2}_1}\pp_x w^{b,0}+u^{b,2}_1\overline{ \pp_xw^{I,0}}+u^{b,2}_1\pp_xw^{b,0} +\overline{u^{I,1}_1}\pp_x w^{b,1}+u^{b,1}_1\overline{ \pp_xw^{I,1}}+u^{b,1}_1\pp_xw^{b,1}\\
&+u^{b,2}_2\overline{\pp_yw^{I,0}}+\left(\overline{u^{I,3}_2}+u^{b,3}_2 \right)\pp_zw^{b,0} +\left(\overline{u^{I,2}_2}+u^{b,2}_2 \right)\pp_zw^{b,1}+4\zeta w^{b,2}+2\zeta\left( \pp_zu^{b,3}_1-\pp_xu^{b,2}_2\right)\\
& -\pp_x^2w^{b,0}-\pp_z^2w^{b,2}+\overline{\pp_yu^{I,1}_1}z\pp_xw^{b,0}+u^{b,1}_1\overline{\pp_y\pp_xw^{I,0}}z+\overline{\pp_yu^{I,0}_1}z\pp_xw^{b,1}+\overline{\pp_yu^{I,2}_2}z\pp_zw^{b,0}\\
&+\overline{\pp_yu^{I,1}_2}z\pp_zw^{b,1}
+\frac{1}{6}\overline{\pp_y^3u^{I,0}_2}z^3\pp_zw^{b,0}+\frac{1}{2}\overline{\pp_y^2u^{I,0}_1}z^2\pp_xw^{b,0}+\frac{1}{2}\overline{\pp_y^2u^{I,0}_2}z^2\pp_zw^{b,1}+\frac{1}{2}\overline{\pp_y^2u^{I,1}_2}z^2\pp_zw^{b,0} =0.
    \end{split}
\end{equation}
Using \eqref{eq_divfree_b} and \eqref{eq_u_b_2}, we have 
\begin{equation}\label{eq_u_b_3_2}
    u_2^{b,3} = \int_z^{\infty}\partial_x u_1^{b,2}\mathrm{d}z = 2\int_z^{\infty}\int_\tau^{\infty} \partial_xw^{b,1}(x,s,t)\mathrm{d}s\mathrm{d}z,
\end{equation}
which together with \eqref{con1} implies 
\begin{equation}\label{eq_u_2_I_3_boun}
    \begin{split}
        \overline{u^{I,3}_2} +u_2^{b,3} =\,& -
    2\int_0^{\infty}\int_\tau^{\infty} \partial_xw^{b,1}(x,s,t)\mathrm{d}s\mathrm{d}z 
    +2\int_z^{\infty}\int_\tau^{\infty} \partial_xw^{b,1}(x,s,t)\mathrm{d}s\mathrm{d}z\\ 
    =\,& -
    2\int_0^{z}\int_\tau^{\infty} \partial_xw^{b,1}(x,s,t)\mathrm{d}s\mathrm{d}z. 
    \end{split}
\end{equation}
Combining  \eqref{initial}, \eqref{con1}, \eqref{eq_w_b_2_orign}, \eqref{eq_u_b_3_2} and \eqref{eq_u_2_I_3_boun}, we find that  $w^{b,2}$  satisfies problem \eqref{wb2}. The proof is complete.
\end{proof}
\noindent {\bfseries Step 5. Some higher order profiles.}
\begin{lemma}\label{lem_b_higher}
    The   inner profiles $u^{b,3}$  and $u^{b,4}_2$  satisfy    
    \begin{gather}\label{A_eq_u_b3_1}
        u^{b,3}_1 = 2\int_z^{\infty}w^{b,2} \mathrm{d}z - \frac{1}{\zeta} u_1^{b,1} - \frac{1}{\zeta}\int_z^{\infty}  \int_\tau^\infty\left( \zeta \partial_x^2u_1^{b,1}  -  \partial_t u^{b,1}_1\right)\mathrm{d}\tau \mathrm{d}z,\\ 
        u^{b,3}_2 =  2\int_z^\infty \int_\tau^{+\infty} \partial_xw^{b,1}(x,s,t)\mathrm{d}s\mathrm{d}\tau, \label{A_eq_u_b3_2}
    \end{gather}
    and 
\begin{equation}\label{A_eq_u_b4_2} 
           u^{b,4}_2 = 2\int_z^{\infty}\int_\tau^\infty \partial_x w^{b,2}\mathrm{d}\tau \mathrm{d}z - \frac{1}{\zeta}\int_z^{\infty} \partial_x u_1^{b,1}\mathrm{d}z - \frac{1}{\zeta}\int_z^{\infty}  \int_\tau^\infty \int_s^\infty\left( \zeta \partial_x^3u_1^{b,1}  -  \partial_t\partial_x u^{b,1}_1\right)\mathrm{d}s\mathrm{d}\tau \mathrm{d}z.  
\end{equation}
\end{lemma}
\begin{proof} Using \ref{eq_divfree_b} and Corollary \eqref{coro_u_b_2}, we have 
\begin{equation}\label{eq_u_b3_22}
   u^{b,3}_2 = \int_z^{+\infty} \partial_xu_1^{b,2}(x,s,t)\mathrm{d}s = 2\int_z^\infty \int_\tau^{+\infty} \partial_xw^{b,1}(x,s,t)\mathrm{d}s\mathrm{d}\tau. 
\end{equation}
From \eqref{eq_u_w_3}, we have 
    \begin{equation}\label{eq_u_b3_11}
         \begin{split}
             u^{b,3}_1  =\,& - \int_z^{\infty}\left[ \frac{1}{\zeta}\int_\tau^\infty\left( \zeta \partial_x^2u_1^{b,1}  -  \partial_t u^{b,1}_1\right)\mathrm{d}\tau -\frac{1}{\zeta} \partial_z  u_1^{b,1} -2w^{b,2} \right]\mathrm{d}z\\ 
             =\,&  2\int_z^{\infty}w^{b,2} \mathrm{d}z - \frac{1}{\zeta} u_1^{b,1} - \frac{1}{\zeta}\int_z^{\infty}  \int_\tau^\infty\left( \zeta \partial_x^2u_1^{b,1}  -  \partial_t u^{b,1}_1\right)\mathrm{d}\tau \mathrm{d}z, 
         \end{split}
    \end{equation}
    which,  together with \ref{eq_divfree_b}, implies that 
    \begin{equation}\label{eq_u_b4_22}
       \begin{split}
           u^{b,4}_2 =\,& \int_z^{+\infty} \partial_xu_1^{b,3}(x,s,t)\mathrm{d}s\\ 
            =\,& 2\int_z^{\infty}\int_\tau^\infty \partial_x w^{b,2}\mathrm{d}\tau \mathrm{d}z - \frac{1}{\zeta}\int_z^{\infty} \partial_x u_1^{b,1}\mathrm{d}z - \frac{1}{\zeta}\int_z^{\infty}  \int_\tau^\infty \int_s^\infty\left( \zeta \partial_x^3u_1^{b,1}  -  \partial_t\partial_x u^{b,1}_1\right)\mathrm{d}s\mathrm{d}\tau \mathrm{d}z. 
       \end{split}
\end{equation}
From \eqref{eq_u_b3_22}--\eqref{eq_u_b4_22}, one can complete the proof.
\end{proof}
\section{Expressions of some source terms.}\label{Appendix B}
In this section, we present
the complete expressions of some source terms in $\eqref{error equations}$, i.e.,

\begin{align}\label{source term F}
    \notag    -F=~&\e\pp_t  u^{b,2}+\e^\frac{3}{2}\pp_t  u^{b,3}+ \e^\frac{3}{2}\pp_t  S+\e^2\pp_t   (0,u^{b,4}_2)^\top\\ \notag
    &+\e^\frac{1}{2}u^{I,0}_1\pp_xu^{b,1}+\left(u^{I,0}_2-\e^\frac{1}{2}\overline{\pp_y u^{I,0}_2}z\right)\pp_z u^{b,1}\\ \notag
    &+\e u^{I,0}_1\pp_x u^{b,2}+\e^\frac{1}{2}u^{I,0}_2\pp_z u^{b,2}+\e^\frac{3}{2}u^{I,0}\cdot\nabla\left(u^{b,3}+S+\e^\frac{1}{2} (0,u^{b,4}_2)^\top\right)\\ \notag
    &+\e^\frac{1}{2}u^{I,1}\cdot\nabla\left(\e u^{I,2}+\e u^{b,2}+\e^\frac{3}{2}u^{b,3}+\e^\frac{3}{2}S+\e^2  (0,u^{b,4}_2)^\top\right)\\
    &+\e^\frac{1}{2}u^{b,1}_1\left(\pp_xu^{I,0}-\overline{\pp_xu^{I,0}}\right)+\e^\frac{1}{2}u^{b,1}_1\pp_x\left(u^a-u^{I,0}\right)\\ \notag
    &+\left(\e u^{b,2}+\e^\frac{3}{2}u^{b,3}+\e^\frac{3}{2}S+\e^2 (0,u^{b,4}_2)^\top\right)\cdot\nabla u^a+\e\left(\begin{matrix}
        \pp_x p^{b,2} \\ 
        0
    \end{matrix}\right)\\ \notag
    &-\e\left(\e^\frac{1}{2}\Delta u^{I,1}+\e^\frac{1}{2}\pp_x^2 u^{b,1}+\e\Delta u^{I,2}+\e\Delta u^{b,2}+\e^\frac{3}{2}\Delta u^{b,3}+\e^\frac{3}{2}\Delta S+\e^2\Delta  (0,u^{b,4}_2)^\top \right)\\ \notag
    &-\zeta\left(\e\pp_x^2 u^{b,2}+\e^\frac{3}{2}\pp_x^2 u^{b,3}+\e^\frac{3}{2}\Delta S+\e^2 \Delta (0,u^{b,4}_2)^\top \right)
    +2\zeta\e\left(\begin{matrix}
        0 \\ 
        \pp_xw^{b,2}
    \end{matrix}\right) \notag\\ \notag
   =&-\sum_{i=1}^{8}F_i. 
\end{align}
Finally, for the components of $G$, we have
\begin{align}\label{source term G}
    \notag -G=&\left( u^{I,0}_1-\overline{ \pp_yu^{I,0}_1}\e^\frac{1}{2} z-\frac{1}{2}\overline{ \pp_y^2 u^{I,0}_1}\e z^2\right)\pp_xw^{b,0}\\ \notag
    &+\left( u^{I,0}_2-\overline{\pp_yu^{I,0}_2}\e^\frac{1}{2}z-\frac{1}{2}\overline{\pp_y^2u^{I,0}_2}\e z^2-\frac{1}{6}\overline{\pp_y^3u^{I,0}_2}\e^\frac{3}{2}z^3\right)\e^{-\frac{1}{2}}\pp_z w^{b,0}\\ \notag
    &+\e^\frac{1}{2}\left(u^{I,0}_1-\overline{\pp_yu^{I,0}_1}\e^\frac{1}{2}z \right)\pp_x w^{b,1}+\left( u^{I,0}_2-\overline{\pp_yu^{I,0}_2}\e^\frac{1}{2}z-\frac{1}{2}\overline{\pp_y^2u^{I,0}_2}\e z^2\right)\pp_zw^{b,1}\\ \notag
    &+\e u^{I,0}_1\pp_xw^{b,2}+\e^\frac{1}{2}\left( u^{I,0}_2-\overline{\pp_yu^{I,0}_2}\e^\frac{1}{2}z \right)\pp_zw^{b,2}+\e^\frac{1}{2} \left( u^{I,1}_1-\overline{ u^{I,1}_1}-\overline{ \pp_yu^{I,1}_1}\e^\frac{1}{2}z \right)\pp_xw^{b,0}\\ \notag
    &+\left(u^{I,1}_2-\overline{ \pp_yu^{I,1}_2}\e^\frac{1}{2}z+\frac{1}{2}\overline{ \pp_y^2u^{I,1}_2}\e z^2 \right)\pp_zw^{b,0}+\e\left( u^{I,1}_1-\overline{u^{I,1}_1}\right)\pp_xw^{b,1}\\ \notag
    &+\e^\frac{1}{2}\left(u^{I,1}_2-\overline{\pp_yu^{I,1}_2}\e^\frac{1}{2}z \right)\pp_zw^{b,1}+\e^\frac{3}{2}u^{I,1}_1\pp_x w^{b,2}+\e u^{I,1}_2\pp_z w^{b,2}\\ \notag
    &+\e^\frac{1}{2}u^{b,1}_1\left( \pp_x w^{I,0}-\overline{ \pp_x w^{I,0}}-\e^\frac{1}{2}\overline{ \pp_y \pp_x w^{I,0}}z\right)+\e^\frac{1}{2}u^{b,1}_1\pp_x\left(\e w^{I,2}+\e w^{b,2}\right)\\ 
    &+\e\left(u^{I,2}_1-\overline{u^{I,2}_1} \right)\pp_xw^{b,0}+\e^\frac{1}{2}\left(u^{I,2}_2-\overline{u^{I,2}_2}-\e^\frac{1}{2}\overline{\pp_yu^{I,2}_2}z \right)\pp_zw^{b,0}+\e^\frac{3}{2} u^{I,2}_1\pp_x w^{b,1}\\ \notag
    &+\e\left(u^{I,2}_2-\overline{ u^{I,2}_2} \right)\pp_zw^{b,1}+ \e u^{I,2}\cdot\nabla\left(\e^\frac{1}{2}w^{I,1}+\e w^{I,2}+\e w^{b,2}\right)\\ \notag
    &+\e u^{b,2}_1\left(\pp_x w^{I,0}-\overline{ \pp_x w^{I,0}} \right)+\e u^{b,2}_2\left(\pp_y w^{I,0}-\overline{ \pp_y w^{I,0}} \right)\\ \notag
    &+\e^\frac{3}{2}u^{b,2}_1\pp_xw^{b,1}+\e u^{b,2}\cdot\nabla \left(\e^\frac{1}{2}w^{I,1}+\e w^{I,2} +\e w^{b,2} \right)+\e^\frac{3}{2}u^{b,3}_1\pp_xw^{b,0}\\ \notag
    &+\e^\frac{3}{2}u^{b,3}\cdot\nabla\left(w^a-w^{b,0}\right)+\left(\e^\frac{3}{2}S+\e^2(0,u^{b,4}_2)^\top\right)\cdot\nabla w^a-\e u^{b,3}_2(x,0,t)\pp_zw^{b,0} \\ \notag
    &-\e \left(  \e^\frac{1}{2}\Delta w^{I,1}+\e^\frac{1}{2}\pp_x^2 w^{b,1}+\e\D w^{I,2}+\e\pp_x^2 w^{b,2}\right)\\ \notag
    &-2\zeta\left(\e^\frac{3}{2}\pp_xu^{b,3}_2-\e^\frac{3}{2}\nabla^\perp\cdot S+\e^2\partial_x u_2^{b,4}\right) \notag\\ \notag
   =&-\sum_{i=1}^{14}G_i. 
\end{align}
 

\section*{Acknowledgement}
The authors would like to thank Professor Huanyao Wen for insightful discussions on the boundary layer theory of complex fluids.
The work of Y. H. Wang was partially supported by the National Natural Science Foundation of China grant 12401274 and the Natural Science Foundation of Hunan Province grant 2024JJ6302. 
\section*{Data availability statement}
\noindent No new data were created or analysed in this study. 
\section*{Conflict of interest} 
\noindent The authors declare that they have no conflict of interest.

\end{document}